\documentclass[a4paper,11pt,reqno]{amsart}

\pdfoutput=1

\usepackage{amsmath, amsfonts, amssymb, amsthm, amscd}
\usepackage[foot]{amsaddr}
\usepackage[a4paper,scale={0.72,0.75},marginratio={1:1},footskip=7mm,headsep=10mm]{geometry}
\usepackage[T1]{fontenc}			
\usepackage[utf8]{inputenc}
\usepackage[italian, english]{babel}
\usepackage{csquotes}
\usepackage[square,numbers]{natbib}	
\usepackage{bbm}
\usepackage{bm}
\usepackage{mathtools}
\usepackage{xspace} % to get the spacing after macros right  
\usepackage{booktabs}
\usepackage{perpage}
\usepackage{enumerate}
\usepackage[pdftitle={Markov Chain control with noise free partial observation},%
pdfauthor={Alessandro Calvia},%
pdftex]{hyperref}
\usepackage{verbatim} % Per commentare sezioni

\newcommand{\cA}{{\ensuremath{\mathcal A}} }
\newcommand{\cB}{{\ensuremath{\mathcal B}} }

\newcommand{\cF}{{\ensuremath{\mathcal F}} }
\newcommand{\cG}{{\ensuremath{\mathcal G}} }

\newcommand{\cJ}{{\ensuremath{\mathcal J}} }

\newcommand{\cM}{{\ensuremath{\mathcal M}} }
\newcommand{\cN}{{\ensuremath{\mathcal N}} }

\newcommand{\cP}{{\ensuremath{\mathcal P}} }

\newcommand{\cT}{{\ensuremath{\mathcal T}} }
\newcommand{\cU}{{\ensuremath{\mathcal U}} }

\newcommand{\cX}{{\ensuremath{\mathcal X}} }
\newcommand{\cY}{{\ensuremath{\mathcal Y}} }
\newcommand{\cZ}{{\ensuremath{\mathcal Z}} }

\newcommand{\bfa}{{\ensuremath{\mathbf a}} }
\newcommand{\bfb}{{\ensuremath{\mathbf b}} }

\newcommand{\bff}{{\ensuremath{\mathbf f}} }

\newcommand{\bfu}{{\ensuremath{\mathbf u}} }

\newcommand{\dd}{\ensuremath{\mathrm d}} 
\newcommand{\dB}{\ensuremath{\mathrm B}} 
\newcommand{\dC}{\ensuremath{\mathrm C}}
\newcommand{\dD}{\ensuremath{\mathrm D}} 
\newcommand{\dL}{\ensuremath{\mathrm L}}

\newcommand{\fru}{\ensuremath{\mathfrak u}}

\newcommand{\R}{\mathbb{R}}

\newcommand{\Z}{\mathbb{Z}}
\newcommand{\N}{\mathbb{N}}

\newcommand{\p}{\ensuremath{\mathrm P}}
\newcommand{\e}{\ensuremath{\mathrm E}}

\newcommand{\ind}{\mathbbm{1}}
\newcommand{\pre}{h^{-1}}

\DeclarePairedDelimiterX{\abs}[1]{\lvert}{\rvert}{#1}
\DeclarePairedDelimiterX{\norm}[1]{\lVert}{\rVert}{#1}
\DeclareMathOperator*{\inter}{\bigcap}       % \sum-like symbol for inter
\newcommand{\Sum}{\sum\limits}
\newcommand{\suptwo}[2]{\sup_{\substack{#1 \\ #2}}} % sup with 2 lines
\newcommand{\sumtwo}[2]{\sum_{\substack{#1 \\ #2}}} % sum with 2 lines
\newcommand{\intertwo}[2]{\inter_{\substack{#1 \\ #2}}} % inter with 2 lines
\DeclareMathOperator{\Int}{int}
\DeclareMathOperator{\dist}{dist}
\DeclareMathOperator{\cco}{\overline{co}}
\DeclareMathOperator{\projX}{proj_X}
\DeclareMathOperator{\projY}{proj_Y}

\newcommand{\ie}{i.\,e. }

\newcommand{\eg}{e.\,g. }

\renewcommand{\epsilon}{\varepsilon}
\let\temp\theta
\let\theta\vartheta
\let\vartheta\temp

\theoremstyle{plain}
\newtheorem{theorem}{Theorem}[section]
\newtheorem{lemma}[theorem]{Lemma}
\newtheorem{proposition}[theorem]{Proposition}
\newtheorem*{proposition*}{Proposition}

\theoremstyle{definition}
\newtheorem{definition}{Definition}[section]
\newtheorem{assumption}{Assumption}[section]

\theoremstyle{remark}
\newtheorem{rem}{Remark}[section]

{\left\lbrace\begin{array}{@{}l@{}}}%
	{\end{array}\right.}

\numberwithin{equation}{section}

\title[Markov chain control with noise-free observation]
	{Optimal control of \mbox{continuous-time} Markov chains with \mbox{noise-free} observation}
\author{A. Calvia$^\star$}
\address{$^\star$University of Milano-Bicocca, Department of Mathematics and its Applications, via R.~Cozzi 55, 20125 Milano (Italy).}
\email{alessandro.calvia@unimib.it}
\thanks{This research was partially supported by GNAMPA-INdAM 2015 and 2016 projects \textit{Applicazioni innovative dei processi di punto marcato} and \textit{Problemi di controllo ottimo con osservazione parziale: applicazioni dei processi di punto marcato} and by MIUR-PRIN 2015 project \textit{Deterministic and stochastic evolution equations}.}
\date{}

\begin{document}
	
\begin{abstract}
	We consider an infinite horizon optimal control problem for a continuous-time Markov chain $X$ in a finite set $I$ with noise-free partial observation. The observation process is defined as $Y_t = h(X_t)$, $t \geq 0$, where $h$ is a given map defined on $I$. The observation is noise-free in the sense that the only source of randomness is the process $X$ itself. The aim is to minimize a discounted cost functional and study the associated value function $V$. After transforming the control problem with partial observation into one with complete observation (the separated problem) using filtering equations, we provide a link between the value function $v$ associated to the latter control problem and the original value function $V$. Then, we present two different characterizations of $v$ (and indirectly of $V$): on one hand as the unique fixed point of a suitably defined contraction mapping and on the other hand as the unique constrained viscosity solution (in the sense of Soner) of a HJB integro-differential equation. Under suitable assumptions, we finally prove the existence of an optimal control.
\end{abstract}

\maketitle
	

\noindent \textbf{AMS 2010:} 93E20, 60J27, 60J25
\section{Introduction}
This paper is concerned with the infinite horizon optimal control of a continuous-time finite-state Markov chain with partial and \mbox{noise-free} observation. The analyzed model is described by a triple of continuous-time stochastic processes $(X,Y,\bfu)=(X_t, Y_t, u_t)_{t \geq 0}$. The process $X$, called \emph{unobserved process}, is the aforementioned Markov chain with values in a finite space $I$ and initial law $\mu$. The process $Y$, called \emph{observed process}, takes values in another finite space $O$. Finally, the process $\bfu$, called \emph{control process}, takes values in the set of Borel probability measures on a compact metric space $U$; we shall require that it belongs to the class $\cU_{ad}$ of predictable processes with respect to the filtration $(\cY_t)_{t \geq 0}$ generated by the process $Y$. This process represents the action of a \emph{relaxed control}, a choice motivated by technical reasons. However, we will be able to recover classical $U$-valued processes, \ie \emph{ordinary controls}, by standard approximation theorems and we will state the main results of this paper using this type of controls.

The purpose of our problem is to control the rate transition matrix (sometimes called \mbox{Q-matrix}) of the unobserved process $X$ via the control process $\bfu$, so that the functional
\begin{equation} \label{eq:introcontrolpb}
	J(\mu, \bfu) = \e_\mu^\bfu \int_0^{+\infty} \int_U e^{-\beta t} f(X_t, \fru) \, u_t(\dd \fru) \, \dd t,
\end{equation} 
is minimized. Here $f$ is a real-valued bounded function, called \emph{cost function}, $\beta$ is a positive \emph{discount factor} and the expectation is taken with respect to a specific probability measure $\p_\mu^\bfu$, depending on the initial law of the Markov chain $X$ and on the control $\bfu$. The infimum of the functional $J$ among all processes in the class $\cU_{ad}$ is the \emph{value function} $V(\mu)$. 

Classically, the solution of this kind of problem requires a two step procedure. First, one needs to write a \emph{filtering equation}, characterizing at each time $t \geq 0$ the conditional law of $X_t$ given $\cY_t$. For our model it can be written as a finite system of scalar equations, whose solution $\pi = (\pi_t)_{t \geq 0}$ coincides with the \emph{filtering process}, \ie it satisfies for each $i \in I$
\begin{equation} \label{eq:introfilteq}
\pi_t(i) = \p_\mu^\bfu(X_t = i \mid \cY_t), \quad \p_\mu^\bfu\text{-a.s.}\,.
\end{equation}
The second step consists in using the filtering process $\pi$ to transform the partial observation problem (\ref{eq:introcontrolpb}) into one with complete observation. This is done through standard computations involving the conditional laws studied in the first step. We associate to the latter problem, with state process coinciding with the filtering process $\pi$, another cost functional and a new cost function $v$ clearly related to the original value function $V$.

The core of this article will be devoted to characterizing $v$ (and indirectly the value function $V$) both as the unique fixed point of a contraction mapping and as a constrained viscosity solution of a \emph{Hamilton-Jacobi-Bellman} equation.
Finally, we will discuss the existence of an optimal ordinary control process $\bfu^\star$, \ie such that the minimum value of the functional $J$ is achieved.

Whilst the definition of the controlled Markov chain that we adopt below (formulated in a weak sense via a change of probability measure) is fairly standard, the specific feature of our model is the fact that the observation is \mbox{noise-free}. With this expression we mean that no external source of randomness is acting on the observed process, but its stochastic nature is due only to the unobserved process $X$.
In our model we suppose that $Y_t = h(X_t)$ for all $t \geq 0$ and for some function $h \colon I \to O$. The function $h$ generates a partition of the set $I$ by its level sets $\pre(a)$, $a \in O$. If at time $t$ one observes $Y_t = a$ for some $a \in O$, then $X_t$ takes some value in the set of states $\pre(a)$. We may say that level sets where $X_t$ lies are observed at any time $t \geq 0$. 

This situation has been studied also under different assumptions on the processes $X$ and $Y$; for instance in the context of an unobserved diffusion process and looking only at the filtering problem without control, research papers as \citep{bryson:linearfilt, crisan:nonlinfilt, runggaldier:filt, takeuchi:lsqestimation} partially deal with this feature and \citep{joannides:nonlinfilt} is devoted entirely to non-linear filtering with noise-free observation (therein called perfect observation). We also mention the chapter devoted to singular filtering in the book by Xiong \citep[Ch. 11]{xiong:stochfilt}. These models, that cannot be analyzed with well established results, have received a sporadic treatment in the literature, despite their potential and useful connection with applications, such as queuing systems (see \eg \citep{asmussen:applprob, bremaud:pp}) and inventory models (see \eg \citep{bensoussan:inventory}). More generally it is worth noticing the connection of our problem with the theory of the so called Hidden Markov Models (see \citep{elliott:hmm} for a comprehensive exposition on this subject). 

Two other works are closely related to the present one, \citep{confortola:filt} and \citep{winter:phdthesis}. In \citep{confortola:filt}, that shares the very same setting as the present one, filtering equations (\ref{eq:introfilteq}) are computed; the filtering process $\pi$ is characterized as a \emph{Piecewise Deterministic Process}, the important class of processes introduced by M.~H.~A.~Davis (see \citep{davis:markovmodels}), and its local characteristics are written down explicitly; however, in \citep{confortola:filt} the process $X$ is not controlled. We will use some results and notation from that paper. We also note an important difference between the approach to PDP optimal control problems presented in \citep{davis:markovmodels} and ours. In the book by Davis the class of control processes is represented by \emph{piecewise open-loop controls}, \ie processes depending only on the time elapsed since the last jump and the position at the last jump time of the PDP. In our problem, instead, we are forced to use a more general class of control policies depending on the past history of jump times and jump positions of the PDP. In fact, as we shall later see, we can find a correspondence between controls for the original problem with partial observation and policies for the reformulated PDP control problem only looking at this larger class. In this sense, an approach closer to ours can be traced in \citep{costadufour:PDPoptcontrol}. There the authors consider an optimal control problem for a PDP (with complete observation), where the control parameter acts only on the jump intensity and on the transition measure of the process but not on its deterministic flow.

A careful comparison with \citep{winter:phdthesis} is required. This PhD thesis analyzes a more general model than ours: alongside the processes $X$ and $Y$ with values in finite spaces, a further finite-state jump process appears, called \emph{environmental}, influencing both the unobserved and the observed processes. Our function $h$ is encoded in the specification of an \emph{information structure}, \ie a partition of the state space $I$. Although in some specific situations our problem can be described in the setting of \citep{winter:phdthesis}, there are some differences, both at level of definitions and of techniques adopted. In our paper, for instance, the initial state $X_0$ of the unobserved process is a random variable with law $\mu$, not just a pre-specified deterministic state; this is a common feature of Markov chains models but it induces some non-trivial complications as we shall see, in particular in connection with the value function $v$ of the reformulated completely observable control problem. We will perform a careful analysis of the connections between the original problem and the completely observable one and we will provide a detailed description of the structure of admissible controls in the reformulated problem: this is required to make the results in \citep{winter:phdthesis} fully rigorous. We will also see that we are able to prove that the value function $v$ is uniformly continuous and we will adopt the concept of constrained viscosity solution, introduced by H.~M.~Soner in \citep{soner:optcontrol1, soner:optcontrol2}, instead of using generalized gradient methods requiring locally Lipschitz continuity of $v$. It is our opinion that this different approach deserves a detailed exposition. It is worth mentioning that to obtain this result a better suited version of the Dynamic Programming Principle is proved. Moreover, a further result proved in this work is the characterization of the value function as the unique fixed point of a specific contraction mapping, under appropriate conditions. Finally, we will provide a direct proof of the existence of an optimal ordinary control, \ie an admissible ordinary control such that the infimum of the functional $J$ is achieved.

The paper is organized as follows: in Section \ref{sec:optcontrformulation} we introduce the notation used throughout the article and make precise the formulation of our problem. The properties of the filtering process are recalled in Section \ref{sec:filteringprocess} and using this process we transform the optimal control problem with partial observation into a complete observation one. In Section \ref{sec:pdpoptcontrol} we properly formulate the complete observation control problem and provide a link between its value function $v$ and the original value function $V$. The previously mentioned characterizations of $v$ are proved in Section \ref{sec:valuefunctioncharacterizations}.
Finally, in Section \ref{sec:existenceoptcontrol} we show the existence of an optimal control.

We mention that the results of this paper have been presented at the \textit{3rd Barcelona Summer School on Stochastic Analysis: a EMS Summer School}, held at the \textit{Centre de Recerca Matem\`atica, Universitat Aut\`onoma de Barcelona}, 27/06--01/07/2016 and at the \textit{International Workshop on BSDEs, SPDEs and their Applications} held at the \textit{University of Edimburgh}, 03-07/07/2017.

\section{The optimal control problem with partial observation} 
\subsection{Formulation} \label{sec:optcontrformulation}
Throughout the paper the set $\N$ denotes the set of natural integers $\N = \{1, 2, \dots \}$, whereas $\N_0 = \{0, 1, \dots \}$. We use also the symbols $\bar \N = \N \cup \{\infty\}$ and $\bar \N_0 = \N_0 \cup \{\infty\}$. As far as measurability is concerned, whenever we write the word \emph{measurable} it is understood that we mean \emph{Borel-measurable}. For a fixed metric space $E$, we denote by $\dB_b(E)$ (resp. $\dC_b(E)$) the set of real valued bounded measurable (resp. bounded continuous) functions on $E$ and by $\cP(E)$ the set of Borel probability measures on $E$.

The aim of our optimal control problem is to optimize the dynamics of a stochastic process $X = (X_t)_{t \geq 0}$, called the \emph{unobserved process} with values in a \emph{state space} $I$. The control is described by another stochastic process $\bfu = (u_t)_{t \geq 0}$, with values in the set of Borel probability measures $\cP(U)$ on a measurable space $(U, \cU)$, the \emph{space of control actions}. This process is chosen in a well defined class and is called \emph{control process}. At any time the chosen control action shall be based on the information provided by a further stochastic process $Y = (Y_t)_{t \geq 0}$, called the \emph{observed process} with values in an \emph{observation space} $O$.

Throughout this paper we will assume that $I$ and $O$ are finite sets and that $U$ is a compact metric space equipped with its Borel $\sigma$-algebra $\cU$. Therefore $\cP(U)$ is a compact metric space, too. We also fix a function $h \colon I \to O$ that, without loss of generality, may be assumed surjective. In general, this function can be constant, but we will exclude this trivial case in what follows.

The unobserved process will be a continuous time homogeneous Markov chain described by a \emph{controlled rate transition matrix} on $I$, sometimes called \mbox{Q-matrix} (see \eg \citep{norris:markovchains}). By this we mean that for each fixed $u \in U$ we have a real square matrix $\Lambda(u) = (\lambda_{ij}(u))_{i,j \in I}$ such that
\begin{enumerate}
	\item $\lambda_{ij}(u) \geq 0$ for all $i, j \in I$, $i \ne j$.
	\item $\sum_{j \in I} \lambda_{ij}(u) = 0$ for all $i \in I$.
\end{enumerate}
It is quite common to write for each $i \in I$
\begin{equation*}
	\lambda_i(u) \coloneqq -\lambda_{ii}(u) = \sumtwo{j \in I}{j \ne i} \lambda_{ij}(u).
\end{equation*}
On these matrix coefficients we introduce the following assumption
\begin{assumption} \label{assumption:lambda}
For each $i, j \in I$ the map $u \mapsto \lambda_{ij}(u)$ is continuous (hence bounded and uniformly continuous). In particular, we have that
\begin{equation*}
	\sup_{u \in U} \lambda_i(u) < +\infty.
\end{equation*}
\end{assumption}

We are going now to build the probability space on which the processes $X$, $Y$, $\bfu$ are defined. To this aim, we choose a canonical setting for the unobserved process $X$, which will be described as a \emph{marked point process}, or MPP for short (see \eg \citep{bremaud:pp}).
Let us define $\Omega$ as the set
\begin{equation*}
	\Omega = \{\omega = (i_0, t_1, i_1, t_2, i_2, \ldots) \colon i_0 \in I, i_n \in I, t_n \in (0, +\infty], t_n < +\infty \Rightarrow t_n < t_{n+1}, n \in \N\}.
\end{equation*}
For each $n \in \N$ we introduce the following random variables 
\begin{align*}
	T_n(\omega) &= t_n; & T_\infty(\omega) &= \lim_{n \to \infty} T_n(\omega); & \xi_0(\omega) &= i_0; & \xi_n(\omega) &= i_n
\end{align*}
and we define the random counting measure
\begin{equation*}
	n((0, t] \times \{i\}) = \sum_{n \in \N} \ind(\xi_n = i) \ind(T_n \leq t), \quad t \geq 0, \, i \in I,
\end{equation*}
with associated natural filtration $\cN_t = \sigma\bigl(n\bigl((0, t] \times \{i\}\bigr), \, 0 \leq s \leq t, \, i \in I\bigr)$.
Finally, let us specify the \mbox{$\sigma$-algebras}
\begin{align*}
	\cX_0 &= \sigma(\xi_0); & \cX_t &= \sigma(\cX_0 \cup \cN_t); & \cX &= \sigma\Bigl(\bigcup_{t \geq 0} \cX_t\Bigr).
\end{align*}

The unobserved process $X$ is defined as
\begin{equation*}
	X_t(\omega) =
	\begin{cases}
		\xi_0(\omega),	& t \in \bigl[0, T_1(\omega)\bigr) \\
		\xi_n(\omega),	& t \in \bigl[T_n(\omega), T_{n+1}(\omega)\bigr), \, n \in \N, \, T_n(\omega) < +\infty \\
		i_\infty,		& t \in \bigl[T_\infty(\omega), +\infty), \, T_\infty(\omega) < +\infty
	\end{cases}
\end{equation*}
where $i_\infty \in I$ is an arbitrary state, that is irrelevant to specify.
Next, we define the observed process $Y$ and its natural filtration $(\cY_t)_{t \geq 0}$ as
\begin{align*}
	Y_t(\omega) &= h(X_t(\omega)), \, t \geq 0, \, \omega \in \Omega; & \cY_t &= \sigma\bigl(Y_s,\, 0 \leq s \leq t\bigr), \, t \geq 0.
\end{align*} 
It is clear that we can equivalently describe this process via a MPP $(\eta_n, \tau_n)_{n \in \N}$ with initial condition $\eta_0 = h(\xi_0) = Y_0$.
Accordingly, the $\sigma$-algebras of the natural filtration of $Y$ are the smallest $\sigma$-algebras generated by the union of $\sigma(\eta_0)$ and the $\sigma$-algebras of the natural filtration of the MPP $(\eta_n, \tau_n)_{n \in \N}$.

As said at the beginning of this section, we need to consider control processes $\bfu$ that are based on the information brought by the observed process $Y$.
More precisely, we will choose controls in the class of \emph{admissible controls}, defined as the set
\begin{equation}\label{eq:admissiblecontrols}
	\cU_{ad} = \Bigl\{\bfu \colon \Omega \times [0,+\infty) \to \cP(U), \, (\cY_t)_{t \geq 0}\text{~--~predictable}\Bigr\}, 
	\quad \cY_t = \sigma\bigl(Y_s,\, 0 \leq s \leq t\bigr).
\end{equation}

\begin{rem}\label{rem:relaxedcontrols}
It must be pointed out that considering $\cP(U)$--valued processes (the so called \emph{relaxed controls}), instead of ordinary $U$--valued processes, has considerable technical benefits that will be fully clear in Section \ref{sec:pdpoptcontrol}. At this stage such a choice has almost no impact on the problem itself (except for a slightly more complicated notation), being both $U$ and $\cP(U)$ compact metric spaces. It is important to notice that any ordinary control is included in this formulation by considering its corresponding process in $\cU_{ad}$ whose value at each time $t \geq 0$ and $\omega \in \Omega$ is given by a probability measure concentrated at a single point in $U$. Ordinary controls are far easier to understand and implement and we will later mention some technical results enabling us to use such controls to prove the main results of this paper. As a final note on this subject, we recall that the existence of an ordinary optimal control will be proved in Section \ref{sec:existenceoptcontrol}.
\end{rem}

Thanks to the peculiar structure of the natural filtration of $Y$ we have a precise characterization of the class $\cU_{ad}$ (see \eg \citep[Lemma 3.3]{jacod:mpp} or \citep[Appendix A2, Theorem T34]{bremaud:pp}).
A control process $\bfu \in \cU_{ad}$ is completely determined by a sequence of \mbox{Borel-measurable} functions $(u_n)_{n \in \bar \N_0}$, with $u_n \colon [0, +\infty) \times O \times \bigl((0, +\infty] \times O \bigr)^n \to \cP(U)$ for each $n \in \bar \N_0$ and we can write
\begin{multline}\label{eq:controlsrepresentation}
u_t(\omega) = u_0(t, Y_0(\omega)) \ind(0 \leq t \leq \tau_1(\omega)) +\\
\sum_{n = 1}^\infty u_n(t, Y_0(\omega), \tau_1(\omega), Y_1(\omega), \dots, \tau_n(\omega), Y_n(\omega)) \ind(\tau_n(\omega) < t \leq \tau_{n+1}(\omega)) + \\
u_\infty(t, Y_0(\omega), \tau_1(\omega), Y_1(\omega), \dots) \ind(t > \tau_\infty(\omega)),
\end{multline}
where $\tau_\infty(\omega) = \lim_{n \to \infty} \tau_n(\omega)$.
This kind of decomposition of a control process $\bfu \in \cU_{ad}$ will be of fundamental importance throughout the paper and we will frequently switch between the notation $(u_t)_{t \geq 0}$ and $(u_n)_{n \in \bar \N_0}$.

The dynamics of the unobserved process will be specified by the initial distribution $\mu$, a probability measure on $I$, and by the following random measure depending on $u$
\begin{equation}\label{eq:predproj}
	\nu^\bfu(\omega;\, \dd t \times \{i\}) = 
	\begin{dcases}
		\ind\bigl(t < T_{\infty}(\omega)\bigr) \int_U \lambda_{X_{t-}(\omega) i}(\fru) \, u_t(\omega \, ; \dd \fru) \, \dd t, &\text{if } i \ne X_{t-}\\
	0, &\text{if } i = X_{t-}
	\end{dcases}
	, \quad \omega \in \Omega, \, \bfu \in \cU_{ad}.
\end{equation}
For sake of simplicity, we will drop $\omega$ in what follows. 

Now set $\p_0$ as the probability measure on $\bigl(\Omega, \cX_0\bigr)$ such that $X_0 = \xi_0$ has law $\mu$.
It is easy to see that the previously described setting is equivalent to that provided in hypothesis (A.2) in \citep{jacod:mpp}.
In fact, one can show that the random measure $\nu^\bfu$ is $(\cX_t)_{t \geq 0}$--predictable and satisfies condition 4 in \citep{jacod:mpp}, \ie
\begin{enumerate}
	\item $\nu^\bfu\bigl(\{t\} \times I\bigr) \leq 1$,
	\item $\nu^\bfu\bigl([T_\infty, +\infty) \times I\bigr) = 0$.
\end{enumerate}
Therefore, by \citep[Th. 3.6]{jacod:mpp}, there exists a unique probability measure $\p_\mu^\bfu$ on $\bigl(\Omega, \cX\bigr)$,
such that $\p_\mu^\bfu \rvert_{\cX_0} = \p_0$ and $\nu^\bfu$ is the $\bigl(\p_\mu^\bfu, \cX_t\bigr)$--predictable projection of $n$.
Once specified the control $\bfu \in \cU_{ad}$ and consequently the probability measure $\p_\mu^\bfu$,
it follows from Assumption \ref{assumption:lambda} and by standard arguments that the point process $n$ is $\p_\mu^\bfu$-a.s. \mbox{non-explosive}, \ie that $T_\infty = +\infty$, \mbox{$\p_\mu^\bfu$-a.s.}.
For this reason we will drop the term $\ind(t < T_\infty)$ appearing in (\ref{eq:predproj}) and, since also $\tau_\infty = +\infty$ \mbox{$\p_\mu^\bfu$-a.s.}, we will avoid specifying the function $u_\infty$ in (\ref{eq:controlsrepresentation}).  

To conclude the previous construction, for a fixed probability measure $\mu$ on $I$ and $\bfu \in \cU_{ad}$ define
\begin{itemize}
	\item $\cX^{\mu, \bfu}$ the \mbox{$\p_\mu^\bfu$-completion} of $\cX$
	($\p_\mu^\bfu$ is extended to $\cX^{\mu, \bfu}$ in the natural way).
	\item $\cZ^{\mu, \bfu}$ the family of elements of $\cX^{\mu, \bfu}$
	with zero $\p_\mu^\bfu$~--~probability.
	\item $\cY_t^{\mu, \bfu} = \sigma(\cY_t, \cZ^{\mu, \bfu})$, for $t \geq 0$.
\end{itemize}
$(\cY_t^{\mu, \bfu})_{t \geq 0}$ is called the \emph{natural completed filtration} of $Y$.

As is common in optimal control problems, control actions are based on some performance criterion.
In our setting we seek to minimize, for all possible choices of the initial distribution $\mu$ of the process $X$, the following \emph{cost functional}
\begin{equation} \label{eq:costfunctional}
	J(\mu, \bfu) = \e_\mu^{\bfu} \biggl[ \int_0^\infty e^{-\beta t} \int_U f(X_t, \fru) \, u_t(\dd \fru) \, \dd t \biggr]
\end{equation}
where $f$ is called \emph{cost function} and $\beta > 0$ is a fixed constant called \emph{discount factor}. In other words, we want to characterize the \emph{value function}
\begin{equation} \label{eq:valuefunction}
	V(\mu) = \inf_{\bfu \in \cU_{ad}} J(\mu, \bfu).
\end{equation}

The following assumption on the cost function $f$ will be in force throughout the paper and ensures that the functional $J$ is well defined (and also bounded).
\begin{assumption} \label{assumption:costfunction}
	The function $f \colon I \times U \to \R$ is continuous. Since $U$ is compact and $I$ finite, $f$ is uniformly continuous and it holds that
	\begin{equation}\label{eq:fbounded}
		\sup_{(i,u) \in I \times U} \abs{f(i,u)} \leq C_f,
	\end{equation}
	for some constant $C_f > 0$.
\end{assumption}
Since $I$ is a finite set, we will denote by $\bff(u)$ the column vector whose components are the values $f(i,u)$ as $i$ varies in the set $I$, for each fixed $u \in U$.

\subsection{The filtering process and the complete observation control problem} \label{sec:filteringprocess}
In this Subsection we use the notation introduced in \citep{confortola:filt} and recall without proof some of the results. In fact, we present generalizations of those contained in the referenced work to the case where the unobserved process is controlled. The proofs are similar and will be omitted.

We can transform the problem formulated above into a complete observation problem by means of another stochastic
process, called the \emph{filtering process} (see \eg \citep{bremaud:pp}), defined for all $i \in I$ as
\begin{equation*}
	\p_\mu^\bfu(X_t = i \mid \cY_t^{\mu, \bfu}), \quad t \geq 0.
\end{equation*}
The filtering process $\pi$ takes values on the set of probability measures on $I$ which can be naturally identified with the canonical simplex on $\R^{\abs{I}}$, \ie
\begin{equation*}
	\Delta = \{\rho \in \R^{\abs{I}} \colon \rho(i) \geq 0, \, \forall i = 1, \dots, \, \abs{I}, \sum_{i = 1}^{\abs{I}} \rho(i) = 1\}.
\end{equation*}
However, in our framework the actual values of $\pi$ lie in the so called \emph{effective simplex} $\Delta_e$.
It is defined as $\Delta_e = \cup_{a \in O} \Delta_a$, where for each $a \in O$ $\Delta_a$ indicates the set of probability measures supported on $\pre{(a)}$.
The effective simplex is a proper subset of $\Delta$ unless the function $h$ is constant. It is obviously compact.
It is worth noticing that the filtering process is a $(\cY_t^{\mu, \bfu})_{t \geq 0}$--adapted process and since $(\cY_t^{\mu, \bfu})_{t \geq 0}$ is right continuous
we can choose a $(\cY_t^{\mu, \bfu})_{t \geq 0}$--progressive version.
We will assume this whenever needed.

In what follows probability measures (or, more generally, finite measures) on $I$ will be identified with row vectors on $\R^{\abs{I}}$.

Let us now define on the effective simplex the vector field 
$F \colon \Delta_e \times U \to \Delta_e$ as
\begin{equation} \label{eq:controlledvectorfield}
	F_j(\nu, u) =
	\begin{cases}
		[\nu \Lambda(u)]_j - (\nu \Lambda(u) \ind_{\pre{(a)}}) \nu_j, 	& j \in \pre{(a)}\\
		0,																& \text{otherwise}
	\end{cases}
	, \quad \nu \in \Delta_a, a \in O.
\end{equation}
Subscript $j$ denotes the \mbox{$j$-th} component of a vector and $\ind_{\pre{(a)}}$ is the column vector
\begin{equation*}
	[\ind_{\pre{(a)}}]_i =
	\begin{cases}
		1, 	& i \in \pre{(a)}\\
		0,	& \text{otherwise}
	\end{cases}
	. 
\end{equation*}
It is clear that the map $u \mapsto F(\nu, u)$ is measurable for all $\nu \in \Delta_e$. Moreover, Assumption \ref{assumption:lambda} implies that $F$ is Lipschitz continuous in $\nu$ uniformly in $u$, \ie there exists a constant $L_F > 0$ such that
\begin{equation}\label{eq:FLipschitz}
	\sup_{u \in U} |F(\nu, u) - F(\rho, u)| \leq L_F |\nu-\rho|, \quad \text{for all } \nu, \rho \in \Delta_a,\, a \in O.
\end{equation}
Therefore, a generalization of \citep[Proposition 2.1]{confortola:filt} provides us with the following result.
\begin{proposition}\label{prop:PDPflow}
	For every $a \in O$, $\rho \in \Delta_a$ and all measurable functions
	$m \colon [0, +\infty) \to \cP(U)$, the differential equation
	\begin{equation} \label{eq:relaxedODE}
	\begin{dcases}
	\frac{\dd}{\dd t} z(t) = \int_U F(z(t), u) \, m(t \, ;\dd u), \quad t \geq 0 \\
	z(0) = \rho
	\end{dcases}
	\end{equation}
	has a unique global solution $y \colon [0, +\infty) \to \R^{\abs{I}}$.
	Moreover $y(t) \in \Delta_a$ for all $t \geq 0$.
\end{proposition}
We will write $\phi_{a, \rho}^m(t)$ instead of $y(t)$ to stress the dependence of the solution to (\ref{eq:relaxedODE}) on $\rho$ and on the measurable function $m$. To ease up the notation a little, we also define $\phi_\rho^m(t) = \phi_{a, \rho}^m(t)$ if $\rho \in \Delta_a$ to denote the global flow associated to the vector field $F$.

As we will see shortly, the filtering process $\pi$ solves a SDE, the \emph{filtering equation}. This SDE can be written pathwise as a system of ODEs; in fact, it is likely, and can be proved, that between two jump times of the process $Y$ the trajectories of $\pi$ have to evolve in a deterministic fashion (in particular, according to the vector field $F$ as we will show), since no new information is brought until a new jump of $Y$ occurs. The stochastic behaviour of $\pi$ resides in the fact that it jumps whenever $Y$ does, \ie the jump times of $\pi$ are the collection $(\tau_n)_{n \in \N}$. The \mbox{post-jump} locations of $\pi$ and its initial value, are determined by the functions $H_a \colon \R^{\abs{I}} \to \R^{\abs{I}}$, mapping row vectors into row vectors and defined for each $a \in O$ as 
\begin{equation} \label{eq:operatorH}
	H_a[\mu](i) =
	\begin{cases}
		0,																		&\text{if } i \notin \pre{(a)},\\
		\frac{\mu(i)}{\mu \ind_{\pre{(a)}}},	&\text{if } i \in \pre{(a)},\quad \mu \ind_{\pre{(a)}} \ne 0,\\
		\nu_a, 																&\text{if } \mu \ind_{\pre{(a)}} = 0, 
	\end{cases}
\end{equation}
where $\nu_a$ is an arbitrary probability measure supported on $\pre{(a)}$ whose exact values are irrelevant.
Whenever the process $Y$ jumps, say to $a \in O$, the process $\pi$ will jump to a specific state prescribed by the function $H_a$. This state belongs to the subset $\Delta_a$ of the effective simplex, necessarily different from the subset of $\Delta_e$ to which the \mbox{pre-jump} state belonged.

We are now ready to state the first important result of this subsection, which extends to the case of controlled process $X$ the corresponding results in \citep{confortola:filt}.
\begin{theorem}[Filtering equation]
	For all $\omega \in \Omega$ define $\tau_0(\omega) \equiv 0$ and for fixed $\bfu \in \cU_{ad}$ the stochastic process
	$\pi^{\mu,\bfu} = (\pi^{\mu,\bfu}_t)_{t \geq 0}$ as the unique solution 
	of the following system of ODEs
	\begin{equation} \label{eq:controlledfilteringprocess}
		\begin{dcases}
			\frac{\dd}{\dd t} \pi^{\mu,\bfu}_t(\omega) = \int_U F(\pi^{\mu,\bfu}_t(\omega), \fru) \, u_t(\omega \, ; \dd \fru), \quad t \in [\tau_n(\omega), \tau_{n+1}(\omega)), \, n \in \N_0 \\
			\pi^{\mu,\bfu}_0(\omega) = H_{Y_0(\omega)}[\mu],\\
			\pi^{\mu,\bfu}_{\tau_n}(\omega) = H_{Y_{\tau_n(\omega)}}\biggl[\pi_{\tau_n^-(\omega)}^{\mu,\bfu}(\omega) \int_U \Lambda(\fru) \, u_{\tau_n^-}(\omega \, ; \dd \fru)\biggr], \, n \in \N.
		\end{dcases}
	\end{equation}
	where $F$ is the vector field defined in (\ref{eq:controlledvectorfield}).
	
	Then, $\pi^{\mu,\bfu}$ is $(\cY_t)_{t \geq 0}$~-~adapted and is a modification of the filtering process,
	\ie 
	\begin{equation*}
		\pi_t^{\mu,\bfu}(i) = \p_\mu^\bfu(X_t = i \mid \cY_t^{\mu, \bfu}), \quad \p_\mu^\bfu\text{--a.s.}, \, t \geq 0, \, i \in I.
	\end{equation*}
\end{theorem}

\begin{rem}
	Thanks to the structure of admissible controls shown in (\ref{eq:controlsrepresentation}) we can write (\ref{eq:controlledfilteringprocess}) as
	\begin{equation*} \label{eq:controlledfilteringprocess2}
		\begin{dcases}
		\frac{\dd}{\dd t} \pi^{\mu,\bfu}_t = \int_U F(\pi^{\mu,\bfu}_t, \fru) \, u_n(t,Y_0, \dots, \tau_n, Y_{\tau_n} \, ; \dd \fru), \quad t \in [\tau_n, \tau_{n+1}), \, n \in \N_0 \\
		\pi^{\mu,\bfu}_0 = H_{Y_0}[\mu],\\
		\pi^{\mu,\bfu}_{\tau_n} = H_{Y_{\tau_n}}\biggl[\pi_{\tau_n^-}^{\mu,\bfu} \int_U \Lambda(\fru) \, u_{n-1}(\tau_n^-,Y_0, \dots, \tau_{n-1}, Y_{\tau_{n-1}} \, ; \dd \fru)\biggr], \, n \in \N.
		\end{dcases}
	\end{equation*}
\end{rem}

From the discussion preceding this theorem, it is obvious (at least heuristically) that the process $\pi$ is a \emph{Piecewise Deterministic Process} (PDP). To characterize a stochastic process as such we need to define a triple $(F, r, R)$, given by the controlled vector field $F$ defined in (\ref{eq:controlledvectorfield}), a controlled jump rate function $r \colon \Delta_e \times U \to [0, +\infty)$ and a controlled stochastic kernel $R$, \ie a probability transition kernel from $(\Delta_e \times U, \cB(\Delta_e) \otimes \cU)$ to $(\Delta_e, \cB(\Delta_e))$.
We define $r$ and $R$ as
\begin{equation} \label{eq:PDPcharacteristictriple}
\begin{split}
	r(\rho, u) &= - \rho \Lambda(u) \ind_{\pre{(a)}}, \quad \rho \in \Delta_a \\
	R(\rho, u; D) &= \sum_{b \in O} \ind_D\bigl(H_b[\rho \Lambda(u)]\bigr) \, q(\rho, u, b), \quad \rho \in \Delta_a \\
	q(\rho, u, b) &=
		\begin{dcases}
		\frac{\rho \Lambda(u) \ind_{\pre(b)}}{- \rho \Lambda(u) \ind_{\pre{(a)}}}
		\ind_{b \ne a}, &\text{if } \rho \Lambda(u) \ind_{\pre{(a)}} \ne 0\\
		q_a(b), &\text{if } \rho \Lambda(u) \ind_{\pre{(a)}} = 0
		\end{dcases}, \quad \rho \in \Delta_a
\end{split}
\end{equation}
where for each $a \in O$ we denote by $q_a = (q_a(b))_{b \in O}$ a probability measure supported on $O \setminus \{a\}$ whose exact values are irrelevant.
It is important to notice that under Assumption \ref{assumption:lambda} $r$ is Lipschitz continuous uniformly in $u$, \ie
\begin{equation} \label{eq:rLipschitz}
\sup_{u \in U} |r(\rho, u) - r(\theta, u)| \leq L_r |\rho - \theta|, \quad \text{for all } \rho, \theta \in \Delta_a,\, a \in O,
\end{equation}
with Lipschitz constant given by $L_r = \sum_{i \in I} \sup_{u \in U} \lambda_i(u)$. We also have that for some $C_r > 0$
\begin{equation} \label{eq:rbounded}
	\sup_{(\rho, u) \in \Delta_e \times U} |r(\rho, u)| \leq C_r.
\end{equation}

\begin{theorem} \label{th:filteringprocessPDP}
	For every $\nu \in \Delta_e$ and all $\bfu \in \cU_{ad}$ 
	the filtering process $\pi^{\nu,\bfu} = (\pi^{\nu,\bfu}_t)_{t \geq 0}$
	defined on the probability space $(\Omega, \cX, \p_\nu^\bfu)$ and taking values in $\Delta_e$
	is a controlled \emph{Piecewise Deterministic Process} with respect to the triple $(F, r, R)$ defined in (\ref{eq:PDPcharacteristictriple}) and with starting point $\nu$.
	
	More specifically, we have that $\p_\nu^\bfu\text{--a.s.}$
	\begin{equation}\label{eq:filterflow}
		\pi_t^{\nu, \bfu} = \phi_{\pi^{\nu, \bfu}_{\tau_n}}^{u_n}(t-\tau_n), \quad t \in [\tau_n, \tau_{n+1}), \, n \in \N_0
	\end{equation}
	\begin{multline}\label{eq:sojourntimes}
		\p_\nu^\bfu(\tau_{n+1} - \tau_n > t \mid \cY^{\nu,\bfu}_{\tau_n}) = \\
		\exp\biggl\{-\int_0^t \int_U r\bigl(\phi_{\pi^{\nu,\bfu}_{\tau_n}}^{u_n(\cdot \, + \, \tau_n)}(s), \fru\bigr) \, u_n(s + \tau_n, Y_0, \dots, \tau_n, Y_{\tau_n} \, ; \dd \fru) \, \dd s\biggr\}, \quad t \geq 0
	\end{multline}
	\begin{multline}\label{eq:postjumplocations}
		\p_\nu^\bfu(\pi^{\nu, \bfu}_{\tau_{n+1}} \in D \mid \cY^{\nu,\bfu}_{\tau_{n+1}^-}) = \\
		\int_U R\bigl(\phi_{\pi^{\nu, \bfu}_{\tau_n}}^{u_n(\cdot \, + \, \tau_n)}(\tau_{n+1}^- - \tau_n), \fru; D\bigr) \, u_n(\tau_{n+1}^-, Y_0, \dots, \tau_n, Y_{\tau_n} \, ; \dd \fru), \quad D \in \cB(\Delta_e)
	\end{multline}
	where, for each $n \in \N_0$, $\phi_{\pi^{\nu, \bfu}_{\tau_n}}^{u_n}$ is the flow starting from $\pi^{\nu, \bfu}_{\tau_n}$ and determined by the controlled vector field $F$ under the action of the control function $u_n(\cdot, Y_0, \dots, \tau_n, Y_{\tau_n})$.
\end{theorem}

\begin{rem}
	The importance of characterizing the filtering process as a Piecewise Deterministic Process (PDP for short) will be clear in the remainder of the paper. This class of processes, introduced by M.H.A.~Davis (see \eg \citep{davis:markovmodels}), has been widely studied in recent years also in connection with optimal control problems. In these problems it is customary to define the class of admissible controls as \emph{piecewise open-loop controls}. These control functions, first studied by Vermes in \citep{vermes:optcontrol}, depend at any time $t \geq 0$ on the position of the PDP at the last jump-time prior to $t$ and on the time elapsed since the last jump. 
	
	In (\ref{eq:admissiblecontrols}) we specified a different class of admissible controls, more suited to our problem and imposed by the fact that we are dealing with partial observation, hence equations (\ref{eq:filterflow}), (\ref{eq:sojourntimes}) and (\ref{eq:postjumplocations}) are changed with respect to the standard formulation with piecewise open-loop controls.
	Another element in contrast with the usual definition of a PDP is the absence in our model of a boundary, since this will be enough for our purposes.
\end{rem}

A common assumption in PDP optimal control problems is that the transition measure $R$ is a Feller kernel. This fails to happen in our situation but, nonetheless, a weaker form of this property holds and it is stated in the following Proposition.
\begin{proposition}\label{prop:weakFellerR}
	Let Assumption \ref{assumption:lambda} hold. Then for every bounded and continuous function $w \colon \Delta_e \to \R$ the function $\rho \mapsto r(\rho, u) \int_{\Delta_e} w(p) R(\rho, u; \dd p)$ is continuous on $\Delta_e$ uniformly in $u \in U$.
\end{proposition}
\begin{proof}
	Fix $\rho \in \Delta_e$, \ie $\rho \in \Delta_a$ for some $a \in O$, and $u \in U$. Let $(\rho_n)_{n \in \N}$ be a sequence such that $\rho_n \to \rho$ as $n \to +\infty$. Without loss of generality we can assume that $\rho_n \in \Delta_a$ for all $n \in \N$. 
	
	Let us consider, first, the case where $r(\rho,u) > 0$. It is easy to see that the function $\rho \mapsto r(\rho,u)$ is continuous on $\Delta_e$ uniformly in $u \in U$, therefore $r(\rho_n,u) > 0$ apart from a finite number of indices $n \in \N$. We want to prove that
	\begin{equation*}
	\biggl| r(\rho_n,u)\int_{\Delta_e} w(p) R(\rho_n,u;\dd p) - r(\rho,u)\int_{\Delta_e} w(p) R(\rho,u;\dd p) \biggl| \to 0, \quad \text{as } n \to +\infty,
	\end{equation*}
	uniformly with respect to $u \in U$, \ie
	\begin{equation}\label{eq:weakFellerRclaim}
	\biggl| \sum_{a \neq b \in O} \sum_{i \in \pre(a)} \sum_{j \in \pre(b)} \Bigl\{w(H_b[\rho_n \Lambda^u]) \rho_i^n - w(H_b[\rho \Lambda^u]) \rho_i \Bigr\} \lambda_{ij}(u) \biggr| \to 0,
	\end{equation}
	as $n \to +\infty$ uniformly in $u \in U$. 
	
	Let $B = \{b \in O, b \neq a, \text{ such that }\Sum_{i \in \pre(a)} \Sum_{j \in \pre(b)} \rho_i \lambda_{ij}(u) > 0 \}$ and $B_0 = B^c \setminus \{a\}$. Clearly, reasoning as before, we have that $\Sum_{i \in \pre(a)} \Sum_{j \in \pre(b)} \rho_i^n \lambda_{ij}(u) > 0$, apart from a finite number of indices $n \in \N$, for each $b \in B$. Moreover $\Sum_{i \in \pre(a)} \Sum_{j \in \pre(b)} \rho_i^n \lambda_{ij}(u) \to 0$ for all $b \in B_0$ uniformly in $u$. Then from equation (\ref{eq:weakFellerRclaim}) we can get the estimate 
	\begin{equation}\label{eq:weakFellerRestimate}
	\begin{split}
	&\biggl| \sum_{b \in B} \sum_{i \in \pre(a)} \sum_{j \in \pre(b)} \Bigl\{w(H_b[\rho_n \Lambda^u]) \rho_i^n - w(H_b[\rho \Lambda^u]) \rho_i \Bigr\} \lambda_{ij}(u) \\ 
	&\qquad + \sum_{b \in B_0} \sum_{i \in \pre(a)} \sum_{j \in \pre(b)} w(H_b[\rho_n \Lambda^u]) \rho_i^n \lambda_{ij}(u) \biggr| \\
	\leq &\, C_\Lambda \sum_{b \in B} \bigl| w(H_b[\rho_n \Lambda^u]) - w(H_b[\rho \Lambda^u]) \bigr| + |B|C_\Lambda \sup_{p \in \Delta_e}|w(p)| \sum_{i \in \pre(a)} \bigl| \rho_i^n - \rho_i \bigr| \\
	&\qquad + \sup_{p \in \Delta_e}|w(p)| \sum_{b \in B_0} \sum_{i \in \pre(a)} \sum_{j \in \pre(b)} \rho_i^n \lambda_{ij}(u),
	\end{split}
	\end{equation}
	where $C_\Lambda = \max_{i \in \pre(a)} \sup_{u \in U} \lambda_i(u)$ is finite thanks to Assumption \ref{assumption:lambda}.
	
	It is clear that the second and the third summand of the last inequality tend to $0$, as $n$ goes to infinity, uniformly in $u$. The difficult task is to show that this is the case also for the first summand. Since the function $w$ is continuous on $\Delta_e$, there exists a modulus of continuity $\eta_w$ such that
	\begin{equation}\label{eq:modulusetaHb}
	\sum_{b \in B} \bigl| w(H_b[\rho_n \Lambda^u]) - w(H_b[\rho \Lambda^u]) \bigr| \leq \sum_{b \in B} \eta_w(\bigl| H_b[\rho_n \Lambda^u] - H_b[\rho \Lambda^u] \bigr|).
	\end{equation}
	Therefore we can fix $b \in B$ and concentrate ourselves on the term $\bigl| H_b[\rho_n \Lambda^u] - H_b[\rho \Lambda^u] \bigr|$. We need to show that
	\begin{multline}\label{eq:operatorHdiff}\textstyle
	\bigl| H_b[\rho_n \Lambda^u] - H_b[\rho \Lambda^u] \bigr| = \sum_{j \in \pre(b)} \Biggl| \frac{\Sum_{i \in \pre(a)} \rho_i \lambda_{ij}(u)}{\Bigl. \Sum_{k \in \pre(a)} \Sum_{l \in \pre(b)} \rho_k \lambda_{kl}(u)\Bigr.} - \frac{\Sum_{i \in \pre(a)} \rho_i^n \lambda_{ij}(u)}{\Bigl. \Sum_{k \in \pre(a)} \Sum_{l \in \pre(b)} \rho_k^n \lambda_{kl}(u)\Bigr. }\Biggr|\\ \textstyle
	= \frac{ \Sum_{j \in \pre(b)} \Bigl| \Sum_{i \in \pre(a)} \Bigl\{ \rho_i \lambda_{ij}(u) \Sum_{k \in \pre(a)} \Sum_{l \in \pre(b)} \rho_k^n \lambda_{kl}(u) - \rho_i^n \lambda_{ij}(u) \Sum_{k \in \pre(a)} \Sum_{l \in \pre(b)} \rho_k \lambda_{kl}(u) \Bigr\}\Bigr|}{\Bigl(\Sum_{k \in \pre(a)} \Sum_{l \in \pre(b)} \rho_k \lambda_{kl}(u) \Bigr) \Bigl( \Sum_{k \in \pre(a)} \Sum_{l \in \pre(b)} \rho_k^n \lambda_{kl}(u) \Bigr)}
	\end{multline}
	tends to $0$ as $n \to +\infty$ uniformly in $u$.
	
	Let $A_0 = \{i \in \pre(a) \text{ such that } \rho_i = 0\}$ and $A = \pre(a) \setminus A_0$. It is obvious that $\Sum_{k \in \pre(a)} \Sum_{l \in \pre(b)} \rho_k \lambda_{kl}(u) = \Sum_{k \in A} \Sum_{l \in \pre(b)} \rho_k \lambda_{kl}(u)$. Using such a decomposition of the set $\pre(a)$ the numerator of (\ref{eq:operatorHdiff}) can be estimated in the following way
	\begin{multline*}
	\Sum_{j \in \pre(b)} \Biggl| \Sum_{i \in \pre(a)} \Biggl\{ \rho_i \lambda_{ij}(u) \Sum_{k \in \pre(a)} \Sum_{l \in \pre(b)} \rho_k^n \lambda_{kl}(u) - \rho_i^n \lambda_{ij}(u) \Sum_{k \in \pre(a)} \Sum_{l \in \pre(b)} \rho_k \lambda_{kl}(u) \Biggr\}\Biggr| \\
	= \Sum_{j \in \pre(b)} \Biggl| \Sum_{i \in A} \rho_i \lambda_{ij}(u) \Sum_{k \in \pre(a)} \Sum_{l \in \pre(b)} \rho_k^n \lambda_{kl}(u) - \Sum_{i \in \pre(a)} \rho_i^n \lambda_{ij}(u) \Sum_{k \in A} \Sum_{l \in \pre(b)} \rho_k \lambda_{kl}(u) \Biggr| \\
	\leq 2 \Sum_{i \in A} |\rho_i - \rho_i^n| \Sum_{j \in \pre(b)} \lambda_{ij}(u) \Sum_{k \in \pre(a)} \Sum_{l \in \pre(b)} \rho_k^n \lambda_{kl}(u) \\
	\qquad \quad + 2 \Sum_{i \in A} \Sum_{k \in A_0} \Sum_{j,l \in \pre(b)} \rho_i^n \rho_k^n \lambda_{ij}(u) \lambda_{kl}(u).
	\end{multline*}
	Therefore we obtain
	\begin{multline} \label{eq:Hbestimate}
	\bigl| H_b[\rho_n \Lambda^u] - H_b[\rho \Lambda^u] \bigr| \\
	\leq 2 \frac{\Sum_{i \in A} |\rho_i - \rho_i^n| \Sum_{j \in \pre(b)} \lambda_{ij}(u)}{\Bigl. \Sum_{i \in A} \Sum_{j \in \pre(b)} \rho_i \lambda_{ij}(u) \Bigr.} + 2 \frac{\Sum_{i \in A} \Sum_{k \in A_0} \Sum_{j,l \in \pre(b)} \rho_i^n \rho_k^n \lambda_{ij}(u) \lambda_{kl}(u)}{\Bigl. \Sum_{i \in A} \Sum_{k \in A} \rho_i \rho_k^n \Sum_{j,l \in \pre(b)}\lambda_{ij}(u) \lambda_{kl}(u) \Bigr.}. 
	\end{multline}	
	
	We are left to prove that the two terms appearing in (\ref{eq:Hbestimate}) tend to zero uniformly in $u$. It suffices to rewrite them in a suitable way, exploiting the properties granted by the decomposition $\pre(a) = A \cup A_0$.
	As for the first summand:
	\begin{multline*}
	\frac{\Sum_{i \in A} |\rho_i - \rho_i^n| \Sum_{j \in \pre(b)} \lambda_{ij}(u)}{\Bigl. \Sum_{i \in A} \Sum_{j \in \pre(b)} \rho_i \lambda_{ij}(u)\Bigr.} = 
	\Sum_{i \in A} \frac{|\rho_i - \rho_i^n| \Sum_{j \in \pre(b)} \lambda_{ij}(u)}{\Bigl. \rho_i \Sum_{j \in \pre(b)} \lambda_{ij}(u)\Bigr.} \, \underbrace{\frac{\rho_i \Sum_{j \in \pre(b)} \lambda_{ij}(u)}{\Bigl. \Sum_{i \in A} \Sum_{j \in \pre(b)} \rho_i \lambda_{ij}(u)\Bigr.}}_{\leq 1 \, \forall u \in U} \\
	\leq \Sum_{i \in A} \frac{|\rho_i - \rho_i^n|}{\rho_i} \to 0 \text{ uniformly in } u \in U. 
	\end{multline*}
	Finally, for the second summand
	\footnote{Provided that $\Sum_{j,l \in \pre(b)} \lambda_{ij}(u) \lambda_{kl}(u) \neq 0$ for all $i \in A$ and all $k \in A_0$. If this is not the case for some indices $i,k$, one can just exclude these indices from the sum appearing in the numerator of this term}:
	\begin{multline*}
	\frac{\Sum_{i \in A} \Sum_{k \in A_0} \Sum_{j,l \in \pre(b)} \rho_i^n \rho_k^n \lambda_{ij}(u) \lambda_{kl}(u)}{\Bigl. \Sum_{i \in A} \Sum_{k \in A} \rho_i \rho_k^n \Sum_{j,l \in \pre(b)} \lambda_{ij}(u) \lambda_{kl}(u)\Bigr.} = \\ 
	\Sum_{i \in A} \Sum_{k \in A_0} \frac{\rho_i^n \rho_k^n \Sum_{j,l \in \pre(b)}  \lambda_{ij}(u) \lambda_{kl}(u)}{\Bigl. \rho_i \rho_i^n \Sum_{j,l \in \pre(b)} \lambda_{ij}(u) \lambda_{kl}(u)\Bigr.} \, \underbrace{\frac{\rho_i \rho_i^n \Sum_{j,l \in \pre(b)} \lambda_{ij}(u) \lambda_{kl}(u)}{\Bigl. \Sum_{i \in A} \Sum_{k \in A} \rho_i \rho_k^n \Sum_{j,l \in \pre(b)} \lambda_{ij}(u) \lambda_{kl}(u)\Bigr.}}_{\leq 1 \, \forall u \in U} \\
	\leq \Sum_{i \in A} \Sum_{k \in A_0} \frac{\rho_k^n}{\rho_i} \to 0 \text{ uniformly in } u \in U. 
	\end{multline*}
	Combining the result just obtained with equations (\ref{eq:weakFellerRestimate}) and (\ref{eq:modulusetaHb}) we get the claim in the case $r(\rho,u) > 0$.
	
	The case $r(\rho,u) = 0$ is much less cumbersome to analyze. Without loss of generality we can assume that the sequence $r(\rho_n,u) \neq 0$ starting from some index $n$ on (the case in which the sequence is equal to $0$ eventually is trivial). We have to prove that
	\begin{equation*}
	\biggl| r(\rho_n,u)\int_{\Delta_e} w(p) R(\rho_n,u;\dd p)\biggl| = \biggl| \sum_{a \neq b \in O} w(H_b[\rho_n \Lambda^u]) \sum_{i \in \pre(a)} \sum_{j \in \pre(b)} \rho_i^n \lambda_{ij}(u) \biggr|
	\end{equation*}
	tends to zero, as $n$ tends to infinity, uniformly with respect to $u \in U$. Thanks to the boundedness of the function $w$ we immediately get
	\begin{multline} \label{eq:caser0estimate}
	\biggl| \sum_{a \neq b \in O} w(H_b[\rho_n \Lambda^u]) \sum_{i \in \pre(a)} \sum_{j \in \pre(b)} \rho_i^n \lambda_{ij}(u) \biggr|\\
	\leq \sup_{p \in \Delta_e} |w(p)| \sum_{a \neq b \in O} \sum_{i \in \pre(a)} \sum_{j \in \pre(b)} \rho_i^n \lambda_{ij}(u).
	\end{multline}
	The properties of the matrix coefficients $\lambda_{ij}(u)$ ensure that
	\begin{equation*}
	0 = \underbrace{\sum_{i \in \pre(a)} \sum_{j \in \pre(a)} \rho_i \lambda_{ij}(u)}_{r(\rho,u)=0} +\sum_{a \neq b \in O} \underbrace{\sum_{i \in \pre(a)} \sum_{j \in \pre(b)} \rho_i \lambda_{ij}(u)}_{\geq 0,\, \forall b \neq a}.
	\end{equation*}
	Therefore the terms $\Sum_{i \in \pre(a)} \Sum_{j \in \pre(b)} \rho_i \lambda_{ij}(u)$ are equal to zero for all $b \neq a$ and since $\rho_i^n \Sum_{j \in \pre(b)} \lambda_{ij}(u) \to \rho_i \Sum_{j \in \pre(b)} \lambda_{ij}(u) = 0$ for all $i \in \pre(a)$ and all $b \neq a$ uniformly in $u$, we get the desired result from equation (\ref{eq:caser0estimate}).
\end{proof}

We can now turn our attention back to the optimal control problem.
We recall that the aim is to minimize the cost functional $J$ defined in (\ref{eq:costfunctional}),
\ie to study the value function $V$ defined in (\ref{eq:valuefunction}).
Since the control processes $\bfu$ are 
$(\cY_t)_{t \geq 0}$--predictable and we know that the filtering process $\pi^{\mu, \bfu}$
provides us with the conditional law of $X_t$ given $\cY_t$, for all $t \geq 0$, it is easy to show that
\begin{equation} \label{eq:costfunctionalmod}
J(\mu, \bfu) = \e_\mu^{\bfu} \biggl[ \int_0^\infty e^{-\beta t} \pi_t^{\mu, \bfu} \int_U \bff(\fru) \, u_t(\dd \fru) \, \dd t \biggr].
\end{equation}

Evidently, this form of the functional $J$ has the advantage of depending on completely observable processes, namely $\pi^{\mu, \bfu}$ and $\bfu$ (which in turn depends on $Y$), so that we have turned the optimal control problem for the Markov chain $X$ into an optimal control problem for the PDP $\pi^{\mu, \bfu}$. Moreover, we can write $J$ in a way that allow us to interpret our problem as a discrete-time control problem. Exploiting the structure of admissible controls $\bfu \in \cU_{ad}$, it is easy to see that
\begin{equation} \label{eq:functionalJdiscrete}
\begin{split}
&J(\mu, \bfu) = \e_\mu^\bfu
\biggl[ \sum_{n = 0}^{+\infty} \int_{\tau_n}^{\tau_{n+1}} e^{-\beta t} \pi^{\mu, \bfu}_t \int_U \bff(\fru) \, u_n(t, Y_0, \tau_1, Y_{\tau_1}, \dots, \tau_n, Y_{\tau_n} \, ; \dd \fru) \, \dd t \biggr] \\
= &\e_\mu^\bfu
\biggl[ \sum_{n = 0}^{+\infty} e^{-\beta \tau_n}\int_0^{\tau_{n+1}-\tau_n} e^{-\beta t} \phi_{\pi^{\mu, \bfu}_{\tau_n}}^{u_n(\cdot \, + \, \tau_n)}(t) \int_U \bff(\fru) \, u_n(t+\tau_n, Y_0, \tau_1, Y_{\tau_1}, \dots, \tau_n, Y_{\tau_n} \, ; \dd \fru)\, \dd t \biggr]\\
= &\sum_{n = 0}^{+\infty} \e_\mu^\bfu \biggl[\e_\mu^\bfu \biggl[
e^{-\beta \tau_n}\!\!\!\int_0^{\tau_{n+1}-\tau_n} \!\!\!e^{-\beta t} \phi_{\pi^{\mu, \bfu}_{\tau_n}}^{u_n(\cdot \, + \, \tau_n)}(t)\!\! \int_U \bff(\fru) u_n(t+\tau_n, Y_0, \dots, \tau_n, Y_{\tau_n} ; \dd \fru)\, \dd t \mid \cY^{\mu,\bfu}_{\tau_n} \biggr] \biggr]\\
= &\e_\mu^\bfu
\biggl[ \sum_{n = 0}^{+\infty} e^{-\beta \tau_n} \int_0^{+\infty} e^{-\beta t} \chi_{\pi^{\mu, \bfu}_{\tau_n}}^{u_n(\cdot \, + \, \tau_n)}(t) \phi_{\pi^{\mu, \bfu}_{\tau_n}}^{u_n(\cdot \, + \, \tau_n)}(t) \int_U \bff(\fru) \, u_n(t+\tau_n, Y_0, \dots, \tau_n, Y_{\tau_n} \, ; \dd \fru) \, \dd t \biggr]\\
= &\e_\mu^\bfu
\biggl[\sum_{n = 0}^{+\infty} e^{-\beta \tau_n} g\bigl(\pi^{\mu,\bfu}_{\tau_n}, u_n(\cdot+\tau_n, Y_0, \tau_1, Y_{\tau_1}, \dots, \tau_n, Y_{\tau_n})\bigr) \biggr]
\end{split}
\end{equation}
where the function $g$ (that will be defined precisely in Section \ref{sec:pdpoptcontrol}) represents the double integral appearing in the fourth line and $\chi_{\pi^{\mu, \bfu}_{\tau_n}}^{u_n(\cdot \, + \, \tau_n)}$ is the survival distribution appearing in (\ref{eq:sojourntimes}).

Unfortunately, the reformulated problem does not fit in the framework of a classical discrete-time optimal control problem (see \eg \citep{bertsekas:stochoptcontrol}) for various reasons. For instance, the problem should be based only on the discrete-time process given by the couples of jump times and jump locations of the filtering process $\pi^{\mu,\bfu}$ (notice that in (\ref{eq:functionalJdiscrete}) also the process $Y$ appears) which, in turn, should not depend on the initial law of the process $X$ and on the control trajectory $\bfu$. Moreover, the class of admissible controls $\cU_{ad}$ is not adequate for a discrete-time problem.
We will solve these issues by reformulating our original control problem into a discrete-time one for the filtering process, introducing also a new class of controls strongly related to the family $\cU_{ad}$. 

\section{The discrete-time PDP optimal control problem} \label{sec:pdpoptcontrol}
In this Section we will reformulate the original optimal control problem into a discrete-time one for the filtering process. This reformulation will fall in the framework of \citep{bertsekas:stochoptcontrol} (from which we will borrow some terminology), a fact that enables us to use known results to study the value function $V$ defined in (\ref{eq:valuefunction}). We will prove that the original control problem and the discrete-time one are deeply connected. In particular, we will show that the value function $V$ can be indirectly characterized by its discrete-time counterpart, that will be analyzed in detail in the next Section.

Let us introduce the \emph{action space}
\begin{equation} \label{eq:actionspace}
\cM = \{m \colon [0, +\infty) \to \cP(U), \text{ measurable}\}
\end{equation}
whose elements are \emph{relaxed controls}. It is known that this space endowed with the \emph{Young topology} is compact (see \eg \citep{davis:markovmodels}).
As already pointed out in Remark \ref{rem:relaxedcontrols}, the set of \emph{ordinary controls}
\begin{equation*}
A = \{\alpha \colon [0, +\infty) \to U, \text{ measurable}\}
\end{equation*}
can be identified as a subset of $\cM$ via the function $t \mapsto \delta_{\alpha(t)}$, $\alpha \in A$, where $\delta_u$ denotes the Dirac probability measure concentrated at the point $u \in U$. As proved in \citep[Lemma 1]{yushkevich:jumpmodel}, this set becomes a \emph{Borel space} when endowed with the coarsest $\sigma$-algebra such that the maps
\begin{equation*}
	\alpha \mapsto \int_0^{+\infty} e^{-t} \psi(t, \alpha(t)) \, \dd t
\end{equation*}
are measurable for all $\psi \colon [0, +\infty) \times U \to \R$, bounded and measurable. This is a fundamental fact to be used in the sequel.
Finally, we define the class of \emph{admissible policies} $\cA_{ad}$ for the \mbox{discrete-time} optimal control problem as
\begin{equation}
 \cA_{ad} = \{\bfa = (a_n)_{n \in \bar \N_0}, a_n \colon \Delta_e \times \bigl((0, +\infty] \times \Delta_e\bigr)^n \to \cM \text{ measurable } \forall n \in \bar \N_0\}.
\end{equation}

We are now ready to introduce the discrete-time PDP optimal control problem. As done in the previous section, we need to put ourselves in a canonical framework for the filtering process and define the following objects.
\begin{itemize}
	\item $\bar{\Omega} = \{\bar{\omega} \colon [0, +\infty) \to \Delta_e, \text{ c\'adl\'ag}\}$ denotes the canonical space for \mbox{$\Delta_e$~--~valued} PDPs. We define $\bar{\pi}_t(\bar{\omega}) = \bar{\omega}(t)$, for $\bar{\omega} \in \bar{\Omega}$, $t \geq 0$, and
	\begin{align*}
		\bar \tau_0(\bar \omega) &= 0, \\
		\bar \tau_n(\bar \omega) &= \inf\{t > \bar \tau_{n-1}(\bar \omega) \text{ s.t. } \bar \pi_t(\bar \omega) \ne \bar \pi_{t^-}(\bar \omega)\}, \quad n \in \N, \\
		\bar \tau_\infty(\bar \omega) &= \lim_{n \to \infty} \bar \tau_n(\bar \omega).
	\end{align*}
	\item The family of $\sigma$-algebras $(\bar{\cF}_t^0)_{t \geq 0}$ given by
	\begin{equation*}
	\bar{\cF}_t^0 = \sigma(\bar{\pi}_s, 0 \leq s \leq t), \quad
	\bar{\cF}^0 = \sigma(\bar{\pi}_s, s \geq  0),
	\end{equation*}
	is the natural filtration of the process $\bar{\pi} = (\bar{\pi}_t)_{t \geq 0}$.
	\item For every $\nu \in \Delta_e$ and all $\bfa \in \cA_{ad}$ we denote by $\bar \p_\nu^\bfa$ the probability measure on $(\bar \Omega, \bar{\cF}^0)$ such that the process $\bar \pi$ is a PDP, starting from the point $\nu$ and with characteristic triple $(F, r, R)$. We this, we mean that $\bar \p_\nu^\bfa$--a.s.
	\begin{equation}\label{eq:flow}
	\bar \pi_t = \phi_{\bar \pi_{\bar \tau_n}}^{a_n}(t-\bar \tau_n), \quad t \in [\bar \tau_n, \bar \tau_{n+1}), \, n \in \N_0.
	\end{equation}
	\begin{multline}\label{eq:sojourntimesPDP}
	\bar \p_\nu^\bfa(\bar \tau_{n+1} - \bar \tau_n > t \mid \bar \cF^0_{\bar \tau_n}) = \\
	\exp\biggl\{-\int_0^t \int_U r(\phi_{\bar \pi_{\bar \tau_n}}^{a_n}(t),\fru) \, a_n(\bar \pi_0, \dots, \bar \tau_n, \bar \pi_{\bar \tau_n})(s \, ; \dd \fru) \, \dd s\biggr\}, \quad t \geq 0.
	\end{multline}
	\begin{multline}\label{eq:postjumplocationsPDP}
	\bar \p_\nu^\bfa(\bar \pi_{\bar \tau_{n+1}} \in D \mid \bar \cF^0_{\bar \tau_{n+1}^-}) = \\
	\int_U R(\phi_{\bar \pi_{\bar \tau_n}}^{a_n}(\bar \tau_{n+1}^- - \bar \tau_n), \fru; D) \, a_n(\bar \pi_0, \dots, \bar \tau_n, \bar \pi_{\bar \tau_n})(\bar\tau_{n+1}^- - \bar\tau_n \, ; \dd \fru),  \quad D \in \cB(\Delta_e).
	\end{multline}
	where, for each $n \in \N_0$, $\phi_{\bar \pi_{\bar \tau_n}}^{a_n}$ is the flow starting from $\bar \pi_{\bar \tau_n}$ and determined by the controlled vector field $F$ under the action of the relaxed control $a_n(\bar \pi_0, \dots, \bar \tau_n, \bar \pi_{\bar \tau_n})$.
	We recall that this probability measure always exists by the canonical construction of a PDP (see \citep[Sec. 24]{davis:markovmodels}).
	\item For every $Q \in \cP(\Delta_e)$ and every $\bfa \in \cA_{ad}$ we define a probability $\bar{\p}_Q^\bfa$ on $(\bar{\Omega}, \bar{\cF^0})$ by 
	$\bar{\p}_Q^\bfa(C) = \int_{\Delta_e} \bar{\p}_\nu^\bfa(C) \, Q(\dd\nu)$ for $C \in \bar{\cF^0}$. This means that $Q$ is the initial distribution of $\bar{\pi}$ under $\bar{\p}_Q^\bfa$.
	\item Let $\bar\cF^{Q, \bfa}$ be the $\bar{\p}_Q^\bfa$-completion of $\bar\cF^0$. We still denote by $\bar{\p}_Q^\bfa$ the measure naturally extended to this new $\sigma$-algebra.
	Let $\bar{\cZ}^{Q, \bfa}$ be the family of sets in $\bar{\cF}^{Q, \bfa}$ with zero $\bar{\p}_Q^\bfa$-probability and define
	\begin{equation*}
	\bar{\cF}_t^{Q, \bfa} = \sigma(\bar{\cF}_t^0 \cup \bar{\cZ}^{Q,\bfa}), \quad 
	\bar\cF_t = \intertwo{Q \in \cP(\Delta_e)}{\bfa \in \cA_{ad}} \bar{\cF}_t^{Q,\bfa}, \quad
	t \geq 0.
	\end{equation*}
	$(\bar{\cF}_t)_{t\geq 0}$ is called the \emph{natural completed filtration} of $\bar{\pi}$.
	By a slight generalization of \citep[Th. 25.3]{davis:markovmodels} it is \mbox{right-continuous}.
\end{itemize}
The PDP $(\bar{\Omega}, \bar{\cF}, (\bar{\cF}_t)_{t \geq 0}, (\bar{\pi}_t)_{t \geq 0}, (\bar{\p}_\nu^\bfa)_{\nu \in \Delta_e}^{\bfa \in \cA_{ad}})$ constructed as above admits the characteristic triple $(F, r, R)$.
For sake of brevity, let us introduce the function $\chi_\rho^m$, depending on $\rho \in \Delta_e$ and $m \in \cM$, given by
\begin{equation}
\chi_\rho^m(t) = \exp\biggl\{-\int_0^t \int_U r(\phi_\rho^m(s), \fru) \, m(s \, ; \dd \fru) \, \dd s\biggr\}, \quad t \geq 0.
\end{equation}
In this way, we can write (\ref{eq:sojourntimesPDP}) as
\begin{equation*}
	\bar \p_\nu^\bfa(\bar \tau_{n+1} - \bar \tau_n > t \mid \bar \cF^0_{\bar \tau_n}) = \chi_\nu^{a_n}(t), \quad t \geq 0.
\end{equation*}
It is worth noticing that $\chi_\rho^m$ solves the ODE
\begin{equation} 
\begin{dcases}
\frac{\dd}{\dd t} z(t) = -z(t) \int_U r(\phi_\rho^m(t), \fru)  \, m(t \, ;\dd \fru), \quad t \geq 0 \\
z(0) = 1
\end{dcases}
.
\end{equation} 

We define the observed process $\bar Y$ on $\bar \Omega$ as follows. Let us introduce the (obviously measurable) function $\projY \colon \Delta_e \to O$ given by
\begin{equation*}
\projY(p) = a, \quad \text{if } p \in \Delta_a, \text{ for some } a \in O
\end{equation*}
and set
\begin{equation*}
	\bar Y_t(\bar\omega) =
	\begin{cases}
		\projY(\bar\pi_0(\bar\omega)),	& t \in \bigl[0, \bar\tau_1(\bar\omega)\bigr) \\
		\projY(\bar\pi_{\bar\tau_n(\bar{\omega})}(\bar\omega)),	& t \in \bigl[\bar\tau_n(\bar\omega), \bar\tau_{n+1}(\bar\omega)\bigr), \, n \in \N, \, \bar\tau_n(\bar\omega) < +\infty \\
		o_\infty,						& t \in \bigl[\bar\tau_\infty(\bar\omega), +\infty), \, \bar\tau_\infty(\bar\omega) < +\infty
	\end{cases}
	\, ,
\end{equation*}
where $o_\infty \in O$ is an arbitrary state, that is irrelevant to specify.
In fact, it is easy to prove by standard arguments that under Assumption \ref{assumption:lambda} for each fixed $\nu \in \Delta_e$ and $\bfa \in \cA_{ad}$ we have that $\bar \tau_\infty = +\infty$, \mbox{$\bar \p_\nu^\bfa$--a.s.}, \ie also in this framework the observed process is \mbox{$\bar \p_\nu^\bfa$--a.s.} non explosive.

Next, let us define the PDP counterpart of the functional $J$, appearing in (\ref{eq:costfunctional}), as follows.
Let $g \colon \Delta_e \times \cM \to \R$ be the \mbox{discrete-time} one-stage cost function defined as
\begin{equation}\label{eq:costfunctiondiscrete}
	g(\nu, m) = \int_0^{+\infty} e^{-\beta t} \chi_\nu^m(t) \phi_\nu^m(t) \int_U \bff(\fru) \, m(t \, ; \dd \fru) \, \dd t.
\end{equation} 
For each $\nu \in \Delta_e$ and $\bfa \in \cA_{ad}$ the PDP cost functional $\bar J$ is defined in analogy with the last line of (\ref{eq:functionalJdiscrete}) as
\begin{equation} \label{eq:pdpcostfunctional}
\bar J(\nu, \bfa) = \bar{\e}_\nu^\bfa 
\biggl[ \sum_{n = 0}^{+\infty} e^{-\beta \bar \tau_n} g\bigl(\bar \pi_{\bar \tau_n}, a_n(\bar \pi_{\bar \tau_0}, \dots, \bar \tau_n, \bar \pi_{\bar \tau_n})\bigr) \biggr].
\end{equation}
Finally, we define the PDP value function as
\begin{equation}\label{eq:pdpvaluefunction}
v(\nu) = \inf_{\bfa \in \cA_{ad}} \bar J(\nu, \bfa).
\end{equation} 

It is now fundamental to establish a connection between the cost functionals (\ref{eq:costfunctionalmod}) and (\ref{eq:pdpcostfunctional}). This link will be given by constructing corresponding admissible controls in $\cU_{ad}$ and admissible policies in $\cA_{ad}$.

\begin{theorem}\label{th:costfunctionalidentif}
	Fix $\mu \in \Delta$ and let $Q \in \cP(\Delta_e)$ the Borel probability measure on $\Delta_e$ concentrated at points $H_a[\mu] \in \Delta_e$, as $a$ varies in the set $O$, defined as
	\begin{equation} \label{eq:probabilityQ}
		Q(D) = \sum_{a \in O} \mu(\pre(a)) \delta_{H_a[\mu]}(D), \quad D \in \cB(\Delta_e).
	\end{equation}
	
	For all $\mu \in \Delta$ and all $\bfu \in \cU_{ad}$ there exists an admissible policy $\bfa \in \cA_{ad}$ such that the laws of $\pi^{\mu, \bfu}$ under $\p_\mu^\bfu$ and of $\bar \pi$ under $\bar \p_Q^\bfa$ are the same. Moreover, for such an admissible policy
	\begin{equation} \label{eq:functionalJequality}
		J(\mu, \bfu) = \sum_{a \in O} \mu(\pre(a)) \bar J(H_a[\mu], \bfa).
	\end{equation}
	Viceversa, for all $\mu \in \Delta$ and all $\bfa = (a_n)_{n \in \N} \in \cA_{ad}$ there exists an admissible control $\bfu \in \cU_{ad}$ such that the same conclusions hold.
\end{theorem}

\begin{proof}
	Let us prove the first part of the Theorem. Let $\bfu \in \cU_{ad}$ be fixed and for all $n \in \N_0$ let us define the functions $a_n \colon \Delta_e \times \bigl((0, +\infty] \times \Delta_e\bigr)^n \to \cM$ as
	\begin{equation*}
		a_n(\nu_0, \dots, s_n, \nu_n)(t \, ; \dd \fru) = u_n\bigl(t + \, s_n, \projY(\nu_0), \dots, s_n, \projY(\nu_n) \, ; \dd \fru \bigr)
	\end{equation*}
	for all possible sequences $(\nu_i)_{i = 0}^n \subset \Delta_e$ and $(s_i)_{i = 1}^n \subset (0, +\infty]$.
	
	Thanks to the fact that $\projY$ is Borel-measurable and that $\cM$ is a Borel space, we can apply \citep[Lemma 3(i)]{yushkevich:jumpmodel} and it follows that each function $a_n$ is measurable. Therefore we have that $\bfa = (a_n)_{n \in \N_0} \in \cA_{ad}$.
	
	The laws of $\pi^{\mu, \bfu}$ under $\p_\mu^\bfu$ and $\bar \pi$ under $\bar \p_Q^\bfa$ are clearly determined by the \mbox{finite-dimensional} distributions of $(\pi_0^{\mu, \bfu}, \tau_1, \pi^{\mu,\bfu}_{\tau_1}, \dots)$ and $(\bar \pi_0, \bar \tau_1, \bar \pi_{\tau_1}, \dots)$ respectively and by the flows associated to the controlled vector fields $F^\bfu$ and $F^\bfa$. These laws, in turn, can be expressed via the initial distributions of $\pi^{\mu, \bfu}_0$ and $\bar \pi_0$ and the conditional distributions of the sojourn times and \mbox{post-jump} locations, \ie for $t \geq 0$ and $D \in \cB(\Delta_e)$ the quantities
	\begin{align}
		&\p_\mu^\bfu(\tau_n - \tau_{n-1} > t, \, \tau_{n-1} < +\infty \mid \pi^{\mu,\bfu}_0, \dots, \tau_{n-1}, \pi^{\mu,\bfu}_{\tau_{n-1}}); \label{eq:sojournlawmccontrol}\\
		&\bar\p_Q^\bfa(\bar\tau_n - \bar\tau_{n-1} > t, \, \bar\tau_{n-1} < +\infty \mid \bar\pi_0, \dots, \bar \tau_{n-1}, \bar{\pi}_{\bar\tau_{n-1}});  \label{eq:sojournlawpdpcontrol} \\
		&\p_\mu^\bfu(\pi^{\mu,\bfu}_{\tau_n} \in D, \, \tau_n < +\infty \mid \pi^{\mu,\bfu}_0, \dots, \pi^{\mu,\bfu}_{\tau_{n-1}}, \tau_n); \label{eq:jumplawmccontrol}\\
		&\bar\p_Q^\bfa(\bar\pi_{\bar\tau_n} \in D, \, \bar\tau_n < +\infty \mid \bar\pi_0, \dots, \bar{\pi}_{\bar\tau_{n-1}}, \bar \tau_n). \label{eq:jumplawpdpcontrol}
	\end{align}
	We will now prove that under the two different probability measures $\p_\mu^\bfu$ and $\bar \p_Q^\bfa$ the distributions (\ref{eq:sojournlawmccontrol})~-~(\ref{eq:jumplawpdpcontrol}) along with the initial laws of $\pi_0^{\mu,\bfu}$ and $\bar{\pi}_0$ are equal.
	
	\noindent \textbf{Initial distribution.} Fix $D \in \cB(\Delta_e)$. Then
	\begin{align*}
		\p_\mu^\bfu(\pi^{\mu,\bfu}_0 \in D) &= 
		\p_\mu^\bfu(H_{Y_0}[\mu] \in D) = 
		\sum_{a \in O} \p_\mu^\bfu(H_a[\mu] \in D, Y_0 = a) \\ 
		&= \sum_{a \in O} \p_\mu^\bfu(Y_0 = a) \delta_{H_a[\mu]}(D) = 
		\sum_{a \in O} \mu(\pre(a)) \delta_{H_a[\mu]}(D) = Q(D)
	\end{align*}
	since the event $\{H_a[\mu] \in D\}$ is either of probability zero or one with respect to $\p_\mu^\bfu$. On the other side
	\begin{align*}
		\bar \p_Q^\bfa(\bar \pi_0 \in D) &=
		\int_{\Delta_e} \bar \p_\nu^\bfa(\bar \pi_0 \in D) \, Q(\dd \nu) =
		\sum_{a \in O} \mu(\pre(a)) \bar \p_{H_a[\mu]}^\bfa(\bar \pi_0 \in D) \\
		&= \sum_{a \in O} \mu(\pre(a)) \delta_{H_a[\mu]}(D) = Q(D).
	\end{align*}
	
	\noindent \textbf{Sojourn times.} Let us analyze first the conditional law
	(\ref{eq:sojournlawmccontrol}). Notice that since we are considering (\ref{eq:sojournlawmccontrol}) on the set ${\tau_{n-1} < +\infty}$, $\pi^{\mu, \bfu}_{\tau_{n-1}}$ is well defined and the law of $\tau_n - \tau_{n-1}$ is not trivial. Fix $p_0, \dots, p_{n-1} \in \Delta_e$, where for each $i = 0, \dots, n-1$, $p_i \in \Delta_{b_i}$ for some $b_0 \ne b_1 \ne \dots \ne b_{n-1} \in O$; fix also $0 < s_1 < \dots < s_{n-1} < +\infty$ (for the same reason as above, we are allowed to take $s_{n-1}$ finite). 
	Since a trajectory of the observed process $Y$ uniquely determines a trajectory of the filtering process $\pi^{\mu,\bfu}$ and viceversa, we can immediately deduce that,
	up to $\p_\mu^\bfu$-null sets
	\begin{align*}
		\cY^{\mu, \bfu}_{\tau_{n-1}} &= \sigma(\pi^{\mu,\bfu}_0, \dots, \tau_{n-1}, \pi^{\mu,\bfu}_{\tau_{n-1}}) & &\text{and} &
		\cY^{\mu, \bfu}_{\tau_n^-} &= \sigma(\pi^{\mu,\bfu}_0, \dots, \pi^{\mu,\bfu}_{\tau_{n-1}}, \tau_n).
	\end{align*}
	From this fact and (\ref{eq:sojourntimes}) we can write for $t \geq 0$
	\begin{align*}
		&\p_\mu^\bfu(\tau_n - \tau_{n-1} > t, \, \tau_{n-1} < +\infty \mid \pi^{\mu,\bfu}_0 = p_0, \dots, \tau_{n-1} = s_{n-1}, \pi^{\mu,\bfu}_{\tau_{n-1}} = p_{n-1})
		\\
		= &\chi_{p_{n-1}}^{u_{n-1}}(t) = \exp\biggl\{ - \int_0^t \int_U r\bigl(\phi_{p_{n-1}}^{u_{n-1}}(s), \fru \bigr) \, u_{n-1}(s + s_{n-1}, \, b_0, \, \dots, \, s_{n-1}, \, b_{n-1} \, ; \dd \fru)\, \dd s \biggr\}.
	\end{align*}
	The function $u_{n-1}=u_{n-1}(\cdot + s_{n-1}, \, b_0, \, \dots, \, s_{n-1}, \, b_{n-1})$ can be clearly expressed as 
	\begin{equation*}
		u_{n-1}(\cdot + s_{n-1}, \, \projY(p_0), \, \dots, \, s_{n-1}, \, \projY(p_{n-1})).
	\end{equation*}
	Therefore, if we compare the previous computation with
	\begin{align*}
		&\bar\p_Q^\bfa(\bar\tau_n - \bar\tau_{n-1} > t, \, \bar\tau_{n-1} < +\infty \mid \bar\pi_0 = p_0, \dots, \bar\tau_{n-1} = s_{n-1}, \bar\pi_{\tau_{n-1}} = p_{n-1}) \\
		= &\chi_{p_{n-1}}^{a_{n-1}}(t)
		= \exp\biggl\{ - \int_0^t \int_U r\bigl(\phi_{p_{n-1}}^{a_{n-1}}(s), \fru \bigr) \, a_{n-1}(p_0, \, \dots, \, s_{n-1}, \, p_{n-1})(s \, ; \dd \fru) \, \dd s \biggr\}
	\end{align*}
	we get the desired result, by definition of $\bfa$.
	
	\noindent \textbf{Post-jump locations.} Continuing with the notation previously introduced (where we add only a new value $s_n$ such that $0 < s_1 < \dots < s_n < +\infty$), we can write (\ref{eq:jumplawmccontrol}) as
	\begin{align*}
	&\p_\mu^\bfu(\pi^{\mu,\bfu}_{\tau_n} \in D, \, \tau_n < +\infty \mid \pi^{\mu,\bfu}_0 = p_0, \dots, \pi^{\mu,\bfu}_{\tau_{n-1}} = p_{n-1}, \tau_n = s_n) \\
	= &\int_U R\bigl(\phi_{p_{n-1}}^{u_{n-1}}(s_n^- - s_{n-1}), \fru \, ; D\bigr) \, u_{n-1}(s_n^-, \, b_0, \, \dots, \, s_{n-1}, \, b_{n-1} \, ; \dd \fru).
	\end{align*}
	On the other hand, we know from (\ref{eq:sojourntimesPDP})
	\begin{align*}
	&\bar\p_Q^\bfa(\bar\pi_{\tau_n} \in D, \, \bar\tau_n < +\infty \mid \bar\pi_0 = p_0, \dots, \bar\pi_{\tau_{n-1}} = p_{n-1}, \bar\tau_n = s_n) \\
	= &\int_U R\bigl(\phi_{p_{n-1}}^{a_{n-1}}(s_n^- - s_{n-1}), \fru \, ; D\bigr) \, a_{n-1}(p_0, \, \dots, \, s_{n-1}, \, p_{n-1})(s_n^- - s_{n-1} \, ; \dd \fru).
	\end{align*}
	Hence again by definition of $\bfa$ we get the equality of the conditional laws (\ref{eq:jumplawmccontrol}) and \ref{eq:jumplawpdpcontrol}.
	
	It remains to prove (\ref{eq:functionalJequality}). Fix $\mu \in \Delta$ and $\bfu \in \cU_{ad}$ with corresponding $\bfa \in \cA_{ad}$ defined as above. Let us define the function $\Phi \colon \bar \Omega \to \R$ as
	\begin{align*}
		\Phi(\bar \omega) &= 
		\sum_{n = 0}^{+\infty} e^{-\beta \bar \tau_n(\bar \omega)} g\bigl(\bar \pi_{\bar \tau_n(\bar \omega)} (\bar \omega), a_n(\bar \pi_0(\bar \omega), \dots, \bar \tau_n(\bar \omega), \bar \pi_{\tau_n(\bar \omega)}(\bar \omega)\bigr)\\
		&= \sum_{n = 0}^{+\infty} e^{-\beta \bar \tau_n(\bar \omega)} g\bigl(\bar \pi_{\bar \tau_n(\bar \omega)} (\bar \omega), u_n(\cdot + \bar\tau_n(\bar \omega), \projY(\bar \pi_0(\bar \omega)), \dots, \bar \tau_n(\bar \omega), \projY(\bar \pi_{\tau_n(\bar \omega)}(\bar \omega))\bigr).
	\end{align*}
	Thanks to Assumptions \ref{assumption:lambda} and \ref{assumption:costfunction} this function is bounded. Since for each $n \in \N_0$ the functions $a_n$ (equivalently $u_n$) are measurable it is also $\bar \cF$-measurable.
	
	Now, take $\bar \omega = \pi^{\mu,\bfu}(\omega)$, $\omega \in \Omega$. It is clear that for all $t \geq 0$ we have $\bar \pi_t(\bar \omega) = \bar \omega(t) = \pi_t^{\mu, \bfu}(\omega)$ and also, by definition of the jump times $(\bar \tau_n)_{n \in \N_0}$, that $\bar \tau_n(\bar \omega) = \tau_n(\omega)$, \mbox{$\p_\mu^\bfu$-a.s.}. Then, we get that \mbox{$\p_\mu^\bfu$-a.s.}
	\begin{align*}
		\Phi(\pi^{\mu,\bfu}(\omega)) &= \sum_{n = 0}^{+\infty} e^{-\beta \tau_n(\omega)} g\bigl(\pi_{\tau_n(\omega)}^{\mu,\bfu} (\omega), u_n(\cdot + \tau_n(\omega), \dots, \tau_n(\omega), \projY(\pi_{\tau_n(\omega)}^{\mu,\bfu}(\omega))\bigr) \\
		&= \sum_{n = 0}^{+\infty} e^{-\beta \tau_n(\omega)} g\bigl(\pi_{\tau_n(\omega)}^{\mu,\bfu} (\omega), u_n(\cdot + \tau_n(\omega), \dots, \tau_n(\omega), Y_{\tau_n(\omega)}(\omega))\bigr)
	\end{align*}
	hence, comparing this result with (\ref{eq:functionalJdiscrete}) we obtain
	\begin{align*}
		J(\mu, \bfu) &= \int_\Omega \Phi(\pi^{\mu,\bfu}(\omega)) \p_\mu^\bfu(\dd \omega) =
		\int_{\bar\Omega} \Phi(\bar \omega) \bar \p_Q^\bfa(\dd \bar \omega) \\
		&= \sum_{a \in O} \mu(\pre(a)) \int_{\bar\Omega} \Phi(\bar \omega) \bar \p_{H_a[\mu]}^\bfa(\dd \bar \omega) = \sum_{a \in O} \mu(\pre(a)) \bar J(H_a[\mu], \bfa)
	\end{align*}
	by definition of the functional $\bar J$.
	
	To prove the second part of the theorem, fix $\mu \in \Delta$ and $\bfa = (a_n)_{n \in \N} \in \cA_{ad}$. Let us start by defining, for each possible sequence $b_0, b_1, \dots \in O$ and $s_1, \dots \in (0, +\infty]$ the following quantities by recursion for all $n \in \N$
	\begin{align*}
		p_0 &= p_0(b_0) = H_{b_0}[\mu] \\
		p_n &= p_n(b_0, s_1, \dots, s_n, b_n) \\
		&= 
		\begin{dcases}
			H_{b_n}\biggl[\phi_{p_{n-1}}^{a_{n-1}(\cdot)}(s_n^- - s_{n-1}) \int_U \Lambda(\fru) \, a_{n-1}(\cdot)(s_n^- - s_{n-1} \, ; \dd \fru)\biggr], &\text{if } s_1 < \dots < s_n \\
			\rho, &\text{otherwise.}
		\end{dcases}	
	\end{align*}
	Here $a_n(\cdot) = a_{n-1}(p_0, \dots, s_n, p_n)$, $s_0 = 0$ and $\rho \in \Delta_e$ is an arbitrarily chosen value.
	
	For all $n \in \N_0$ we define the functions $u_n \colon [0, +\infty) \times O \times \bigl((0, +\infty] \times O\bigr)^n \to \cP(U)$ as
	\begin{equation*}
	u_n(t, b_0, \dots, s_n, b_n \, ; \dd \fru) = 
	\begin{cases}
		a_n(p_0, \dots, s_n, p_n)(t - s_n \, ; \dd \fru), 	&\text{if } t \geq s_n \\
		\fru,													&\text{if } t < s_n
	\end{cases}
	\end{equation*}
	where $\fru \in U$ is some fixed value that is irrelevant to specify.
	Thanks to the fact that each of the functions $(b_0, \dots, s_n, b_n) \mapsto p_n$ is Borel-measurable and that $\cM$ is a Borel space, we can use \citep[Lemma 3(ii)]{yushkevich:jumpmodel} to conclude that all the $u_n$'s are Borel-measurable and therefore $\bfu = (u_n)_{n \in \N_0} \in \cU_{ad}$.
	
	Similarly to what we did in the proof of the first part of the Theorem, we need to characterize the laws of $\pi^{\mu, \bfu}$ under $\p_\mu^\bfu$ and $\bar \pi$ under $\p_Q^\bfa$. First of all, let us notice that we do not need to prove again that the initial distributions of the two processes are equal since they do not depend on the controls $\bfu$ and $\bfa$. Therefore, we need only to compare the conditional distributions	
	\begin{align*}
	&\p_\mu^\bfu(\tau_n - \tau_{n-1} > t, \, \tau_{n-1} < +\infty \mid Y_0, \dots, \tau_{n-1}, Y_{\tau_{n-1}});\\
	&\bar\p_Q^\bfa(\bar\tau_n - \bar\tau_{n-1} > t, \, \bar\tau_{n-1} < +\infty \mid \bar Y_0, \dots, \bar \tau_{n-1}, \bar{Y}_{\bar\tau_{n-1}}); \\
	&\p_\mu^\bfu(\pi^{\mu,\bfu}_{\tau_n} \in D, \, \tau_n < +\infty \mid Y_0, \dots, Y_{\tau_{n-1}}, \tau_n);\\
	&\bar\p_Q^\bfa(\bar\pi_{\bar\tau_n} \in D, \, \bar\tau_n < +\infty \mid \bar Y_0, \dots, \bar{Y}_{\bar\tau_{n-1}}, \bar \tau_n),
	\end{align*}
	where $t > 0$ and $D \in \cB(\Delta_e)$.
	This can be done in the same way as in the first part of the proof, this time using the definition of the control $\bfu \in \cU_{ad}$ and the obvious fact that, up to $\bar \p_Q^\bfa$-null sets we have
	\begin{align*}
		\bar \cF_{\bar \tau_{n-1}} &= \sigma(\bar Y_0, \dots, \bar \tau_{n-1}, \bar Y_{\bar \tau_{n-1}}) & &\text{and} &
		\bar \cF_{\tau_n^-} &= \sigma(\bar Y_0, \dots, \bar Y_{\bar \tau_{n-1}}, \bar \tau_n).
	\end{align*}
	
	Finally, to prove (\ref{eq:functionalJequality}) it suffices to define $\Phi \colon \bar \Omega \to \R$ as
	\begin{equation*}
		\Phi(\bar \omega) = 
		\sum_{n = 0}^{+\infty} e^{-\beta \bar \tau_n(\bar \omega)} g\bigl(\bar \pi_{\bar \tau_n(\bar \omega)} (\bar \omega), u_n(\cdot + \bar\tau_n, \bar Y_0(\bar \omega), \dots, \bar \tau_n(\bar \omega), \bar Y_{\tau_n(\bar \omega)}(\bar \omega)\bigr).
	\end{equation*}
	and notice that $p_n(\bar Y_0, \dots, \bar\tau_n, \bar Y_{\bar\tau_n}) = \bar \pi_{\bar \tau_n}$, so that we can write 
	\begin{equation*}
		\Phi(\bar \omega) = 
		\sum_{n = 0}^{+\infty} e^{-\beta \bar \tau_n(\bar \omega)} g\bigl(\bar \pi_{\bar \tau_n(\bar \omega)} (\bar \omega), a_n(\bar \pi_0(\bar \omega), \dots, \bar \tau_n(\bar \omega), \bar \pi_{\tau_n(\bar \omega)}(\bar \omega)\bigr).
	\end{equation*}
	The desired equality follows from the same reasoning as in the first part of the proof.
\end{proof}

\begin{rem}\label{rem:stationarypolicies}
The proof of Theorem \ref{th:costfunctionalidentif} provides us with an explicit way to construct an admissible policy $\bfa$ given an admissible control $\bfu$ and viceversa. The case that most concerns us is to build an admissible control $\bfu$ when $\bfa$ is a stationary admissible policy, \ie $\bfa = (a_0, a, a, \dots)$. The function $a_0$ depends on the starting point of the filtering process and $a$ is a function of its jump times and jump locations. In other words, this kind of admissible policy represents a piecewise open-loop control. Notice that here dependency on jump times (and not only on the time elapsed since the last one) must be taken into account. This is a generalization of the original definition by Vermes (cfr. \citep{vermes:optcontrol}).
\end{rem}

Having identified the original problem with the discrete-time PDP problem, we can concentrate our analysis on the latter one. What we are aiming at is to prove that $v$ is the unique fixed point of the operator $\cT \colon \dB_b(\Delta_e) \to \dB_b(\Delta_e)$ defined for all $\nu \in \Delta_e$ as
\begin{multline} \label{eq:operatorT}
	\begin{split}
	\cT w(\nu) 
	&\coloneqq \inf_{m \in \cM} \int_0^\infty \int_U e^{-\beta t} L(\phi_\nu^{m}(t), \chi_\nu^{m}(t), \fru, w) \, m(t; \dd \fru) \, \dd t \\
	&\coloneqq \inf_{m \in \cM} \int_0^\infty \int_U e^{-\beta t} \chi_\nu^{m}(t) \biggl[\phi_\nu^{m}(t) \bff(\fru) +
	\end{split} \\
	r(\phi_\nu^{m}(t), \fru)
	\int_{\Delta_e} w(p) R(\phi_\nu^{m}(t), \fru; \dd p) \biggr] \, m(t; \dd \fru) \, \dd t.
\end{multline}
It is easy to check that under Assumptions \ref{assumption:lambda} and \ref{assumption:costfunction} $\cT$ is a contraction. Therefore, we just need to show that $v$ is a fixed point of $\cT$. To do so, we will resort to results connected with the so called \emph{lower semicontinuous model} of \citep{bertsekas:stochoptcontrol}, that ensure the existence of an optimal non-randomized stationary (Borel-)measurable policy $\bfa \in \cA_{ad}$, in the same sense given in Remark \ref{rem:stationarypolicies} above.

These kind of model require to specify the quadruple $(S, \cM, q, g)$, where $S$ is the state space and $q$ is a transition kernel for the underlying discrete time process. $\cM$ and $g$ are the action space and discrete-time one-stage cost function of (\ref{eq:actionspace}) and (\ref{eq:costfunctiondiscrete}).
Attention must be paid to the fact that the state space of this model does not coincide with the state space $\Delta_e$ of the PDP, since all functions but the first one composing a policy depend also on jump times, as stressed in Remark \ref{rem:stationarypolicies}. Apart from this minor complication, verification of assumptions of the lower semicontinuous model are quite standard. In our setting it all boils down to the next Lemma, whose proof is omitted since it is based on routine computations.
\begin{lemma}\label{prop:fixedpointT}
	Under Assumptions \ref{assumption:lambda} and \ref{assumption:costfunction} we have the following results.
	\begin{enumerate}
		\item The transition kernel $q$ defined for all $\nu \in \Delta_e$, $D \in \cB(\Delta_e)$ and $m \in \cM$ by
		\begin{equation}
		q(D \mid \nu, m) \coloneqq \int_0^{+\infty} e^{-\beta t} \chi_\nu^{m}(t) \int_U r(\phi_\nu^{m}(t), u) R(\phi_\nu^{m}(t), u; D) \, m(t; \dd u) \, \dd t
		\end{equation}
		is continuous.
		\item The cost function $g$ defined in (\ref{eq:costfunctiondiscrete}) is bounded and continuous.
	\end{enumerate}  
\end{lemma}
Since the hypotheses of the lower semicontinuous model of \citep{bertsekas:stochoptcontrol} are verified, we are able to state (details on the proof can be found in \eg \citep[Corollary 9.17.2]{bertsekas:stochoptcontrol}) the following standard result on the existence of an optimal policy and regularity of the value function.
\begin{proposition} \label{prop:optpolicy}
	Under Assumptions \ref{assumption:lambda} and \ref{assumption:costfunction} there exists an optimal policy $\bfa^\star \in \cA_{ad}$, \ie a policy such that 
	$$v(\nu) = \bar J(\nu, \bfa^\star), \quad \text{for all } \nu \in \Delta_e.$$
	Moreover, this policy is stationary, the value function $v$ is lower semicontinuous and it is the unique fixed point of the operator $\cT$.
\end{proposition}

\begin{rem}
	It is worth mentioning that Assumption \ref{assumption:lambda} reveals its fundamental role in the course of the proof of Proposition \ref{prop:optpolicy}. In fact, it ensures the continuity of the function $u \mapsto r(\rho, u) \int_{\Delta_e} w(p) R(\rho, u; \dd p)$ for all $\rho \in \Delta_e$ and all $w \in \dC_b(\Delta_e)$, which is crucial to guarantee continuity of $q$ with respect to $m \in \cM$. However all the other results shown so far remain true even if we weaken Assumption \ref{assumption:lambda} and just ask that the maps $u \mapsto \lambda_{i j}(u)$ are measurable for all $i,j \in I$ and that $\sup_{u \in U} \lambda_i(u) < +\infty$ for all $i \in I$.
\end{rem}

Relaxed controls are difficult to interpret and implement in practice. Fortunately, we are able to show from the Proposition \ref{prop:optpolicy} that $v$ is also the unique fixed point of the operator $\cG \colon \dB_b(\Delta_e) \to \dB_b(\Delta_e)$ given by
\begin{equation}\label{eq:operatorG}
\cG w(\nu) = \inf_{\alpha \in A} \int_0^\infty e^{-\beta t} L(\phi_\nu^{\alpha}(t), \chi_\nu^{\alpha}(t), \alpha(t), w) \, \dd t, \quad \nu \in \Delta_e,
\end{equation}
where the infimum is taken among all possible ordinary control instead of relaxed ones.
Thanks again to Assumptions \ref{assumption:lambda} and \ref{assumption:costfunction}, standard arguments show that $\cG$ is a contraction.
\begin{theorem}
	Under Assumptions \ref{assumption:lambda} and \ref{assumption:costfunction} $v$ is the unique fixed point of the operator $\cG$.
\end{theorem}
\begin{proof}
	It is clear that $v = \cT v \leq \cG v$, so we just need to prove the reverse inequality. We previously saw that there exists a stationary optimal policy $\bfa^\star$ for the discrete time control problem. By \citep[Corollary 9.12.1]{bertsekas:stochoptcontrol} this implies that the infimum in (\ref{eq:operatorT}) is attained for each $\nu \in \Delta_e$ by some $m^\star \in \cM$, with $m^\star = m^\star(\nu)$, and since the set $A$ of ordinary controls is dense in $\cM$ with respect to the Young topology (see \eg \citep[V, Th. 7]{sainte-beuve:relaxeddensity}), we can construct a sequence $(\alpha_n)_{n \in \N} \subset A$ such that $\alpha_n \to m^\star$ as $n \to \infty$. Moreover we have that the function $\cJ(\nu, m) \coloneqq \int_0^\infty \int_U e^{-\beta t} L(\phi_\nu^{m}(t), \chi_\nu^{m}(t), u, v) \, m(t; \dd u) \, \dd t$ is continuous in $m$ for all $\nu \in \Delta_e$ (the computations are similar to those of proposition (\ref{prop:fixedpointT})). Hence we get that for each fixed $\nu \in \Delta_e$
	\begin{equation*}
	\cJ(\nu, \alpha_n) \to \cJ(\nu, m^\star) = \cT v(\nu) = v(\nu).
	\end{equation*}
	Noticing that $\cG v(\nu) \leq \cJ(\nu, \alpha_n)$ for all $n \in \N$, we get the result.
\end{proof}

We can finally provide the link between the two value functions $V$ and $v$.
\begin{theorem} \label{th:valuefunctionsidentif}
	For all $\mu \in \Delta$ we have that
	\begin{equation}
	V(\mu) = \sum_{a \in O} \mu(\pre(a))v(H_a[\mu]).
	\end{equation}
\end{theorem}
\begin{proof}
	Recall that we know from Theorem \ref{th:costfunctionalidentif} that for all $\mu \in \Delta$
	\begin{equation*}
	J(\mu, \bfu) = \sum_{a \in O} \mu(\pre(a)) \bar J(H_a[\mu], \bfa),
	\end{equation*}
	where $\bfu \in \cU_{ad}$ and $\bfa \in \cA_{ad}$ are corresponding admissible controls and admissible policies.
	
	Let now $\mu \in \Delta$ be fixed. It is obvious that $V(\mu) \geq \sum_{a \in O} \mu(\pre(a))v(H_a[\mu])$. In fact, since $\bar J(H_a[\mu], \bfa) \geq v(H_a[\mu])$ for all $\bfa \in \cA_{ad}$ and all $a \in O$, we get that for all $\bfu \in \cU_{ad}$
	\begin{equation*}
	J(\mu, \bfu) \geq \sum_{a \in O} \mu(\pre(a))v(H_a[\mu]),
	\end{equation*}
	and we get the desired inequality by taking the infimum on the left hand side with respect to all $\bfu \in \cU_{ad}$.
	
	The reverse inequality is easily obtained by taking an optimal policy $\bfa^\star \in \cA_{ad}$ (whose existence is guaranteed by Proposition \ref{prop:optpolicy}) and considering its corresponding admissible control $\bfu^\star \in \cU_{ad}$. From Theorem \ref{th:costfunctionalidentif} we immediately get that
	\begin{equation*}
		V(\mu) \leq J(\mu, \bfu^\star) = \sum_{a \in O} \mu(\pre(a)) \bar J(H_a[\mu], \bfa^\star) = \sum_{a \in O} \mu(\pre(a))v(H_a[\mu]). \qedhere
	\end{equation*}
\end{proof}

Theorem \ref{th:valuefunctionsidentif} gives us a way to go back and forth between the original control problem and the discrete-time one. Moreover, we easily deduce that an admissible control $\bfu \in \cU_{ad}$ is optimal if and only if its corresponding admissible policy $\bfa \in \cA_{ad}$ is. In the next Section we will focus our attention on the analysis of the value function $v$, that will indirectly give informations about the original value function $V$.

\section{Characterization of the value function}\label{sec:valuefunctioncharacterizations}
We will characterize the PDP value function $v$ in two ways: first we will study a fixed point problem related to the operator $\cG$. We already know that $v$ is the unique fixed point of $\cG$ as an operator acting on the space of bounded Borel-measurable functions on $\Delta_e$ into itself. What we will prove is that it is the unique fixed point of $\cG$ as an operator acting on the space of continuous functions into itself. Once gained the continuity of $v$ on $\Delta_e$, hence its uniform continuity and boundedness, we will prove that it is also a \emph{constrained viscosity solution} of a HJB equation.

\subsection{The fixed point problem}
Let us denote by $\dC(\Delta_e)$ the space of continuous functions on $\Delta_e$ equipped with the usual sup norm. We recall that, since $\Delta_e$ is a compact subset of $\R^{\abs{I}}$, this is the  space of bounded and uniformly continuous functions on $\Delta_e$. 

To prove continuity of $v$ we need to show that $\cG$ maps the space $\dC(\Delta_e)$ into itself and that $v$ is its unique fixed point in that space (recall that we already established that $\cG$ is a contraction).
We shall also need a version of the \emph{Dynamic Programming Principle} suited to this problem, that we are going to prove.

\begin{proposition}[Dynamic Programming Principle]\label{prop:DPP}
	For all functions $w \in \dB_b(\Delta_e)$ and all $T > 0$ the function $\cG w$ satisfies the following identity
	\begin{equation} \label{eq:DPP}
	\cG w(\nu) = \inf_{\alpha \in A} 
	\biggl\{ \int_0^T e^{-\beta t} L(\phi_\nu^{\alpha}(t), \chi_\nu^{\alpha}(t), \alpha(t), w) \, \dd t + e^{-\beta T} \chi_\nu^{\alpha}(T) \cG w(\phi_\nu^{\alpha}(T)) \biggr\}.
	\end{equation}
\end{proposition}

\begin{rem}
	It is worth noticing that taking $w = v$ we get the standard statement of the Dynamic Programming Principle.
\end{rem}
\begin{proof}
	Let $T > 0$, $w \in \dB_b(\Delta_e)$ and $\nu \in \Delta_e$ be fixed and let us define $\tilde w(\nu)$ the right hand side of (\ref{eq:DPP}).
	
	We will show first that $\cG w(\nu) \leq \tilde w(\nu)$. Choose an arbitrary $\alpha \in A$ and define $\rho \coloneqq \phi_\nu^{\alpha}(T)$. For some fixed $\epsilon > 0$, let $\alpha^\epsilon \in A$ be such that
	\begin{equation}\label{eq:alphaepsilon}
	\cG w(\rho) + \epsilon \geq \int_0^\infty e^{-\beta t} L(\phi_\rho^{\alpha^\epsilon}(t), \chi_\rho^{\alpha^\epsilon}(t), \alpha^\epsilon(t), w) \, \dd t.
	\end{equation} 
	Next, define the function $\tilde \alpha \colon [0,+\infty) \to U$ as
	\begin{equation*}
	\tilde \alpha(t) = \alpha(t) \ind_{[0,T]}(t) + \alpha^\epsilon(t-T) \ind_{(T, +\infty)}(t). 
	\end{equation*}
	It is clearly measurable, \ie $\tilde \alpha \in A$, and it is straightforward to notice that
	\begin{equation*}
	\cG w(\nu) \leq \int_0^T e^{-\beta t} L(\phi_\nu^{\alpha}(t), \chi_\nu^{\alpha}(t), \alpha(t), w) \, \dd t + \int_T^\infty e^{-\beta t} L(\phi_\nu^{\tilde \alpha}(t), \chi_\nu^{\tilde \alpha}(t), \tilde \alpha(t), w) \, \dd t.
	\end{equation*}
	Thanks to the flow property of $\phi$ we have that for $t > T$ the equality $\phi_\nu^{\tilde \alpha}(t) = \phi_\rho^{\alpha^\epsilon}(t-T)$ holds. Moreover, it can be easily shown that $\chi_\nu^{\tilde \alpha}(t) = \chi_\nu^{\alpha}(T) \chi_\nu^{\alpha^\epsilon}(t-T)$, for $t > T$.
	With this in mind and performing a simple change of variables, we get that
	\begin{equation*}
	\int_T^\infty e^{-\beta t} L(\phi_\nu^{\tilde \alpha}(t), \chi_\nu^{\tilde \alpha}(t), \tilde \alpha(t), w) \, \dd t = e^{-\beta T} \chi_\nu^{\alpha}(T) \int_0^\infty e^{-\beta t} L(\phi_\rho^{\alpha^\epsilon}(t), \chi_\rho^{\alpha^\epsilon}(t), \alpha^\epsilon(t), w) \, \dd t.
	\end{equation*}
	Therefore, we have from (\ref{eq:alphaepsilon}) that for all $\epsilon > 0$
	\begin{equation*}
	\cG w(\nu) \leq  \int_0^T e^{-\beta t} L(\phi_\nu^{\alpha}(t), \chi_\nu^{\alpha}(t), \alpha(t), w) \, \dd t + e^{-\beta T} \chi_\nu^{\alpha}(T) \bigl[\cG w(\rho) + \epsilon \bigr].
	\end{equation*}
	Since $\alpha$ is arbitrary, we can take the limit as $\epsilon \to 0^+$ and then the infimum on the set $A$ to get that $\cG w(\nu) \leq \tilde w(\nu)$.
	The reverse inequality is easily obtained with similar computations.
\end{proof}

We provide now an estimate that will be fundamental in proving the next Proposition.
\begin{lemma}\label{lemma:JTwestimate}
	Let $T > 0$ and $w \in \dC(\Delta_e)$ be fixed and define for all $\nu \in \Delta_e$ and all $\alpha \in A$
	\begin{equation}\label{eq:JTwdefinition}
		\cJ_{T,w}(\nu, \alpha) = \int_0^T e^{-\beta t} L(\phi_\nu^{\alpha}(t), \chi_\nu^{\alpha}(t), \alpha(t), w) \, \dd t.
	\end{equation} 
	Then, under Assumptions \ref{assumption:lambda} and \ref{assumption:costfunction}, there exists constants $C, K_1, K_2 > 0$ and a modulus of continuity $\eta$
	\footnote{\ie a continuous, nondecreasing, subadditive function $\eta \colon [0,+\infty) \to [0,+\infty)$ such that $\eta(t) \to 0$ as $t \downarrow 0$.}
	such that for all $\alpha \in A$
	\begin{equation}
		\bigl|\cJ_{T,w}(\nu, \alpha) - \cJ_{T,w}(\rho, \alpha)\bigr| \leq K_1 |\nu - \rho| + K_2 \eta(C |\nu-\rho|).
	\end{equation}	
\end{lemma}
\begin{proof}
	Let $\alpha \in A$ and $\nu \in \Delta_e$ be fixed. It is clear that $\nu \in \Delta_a$ for some $a \in O$. Let us consider a sequence $(\nu_k)_{k \in \N}$ such that $\nu_k \to \nu$ as $k \to +\infty$. Without loss of generality we can take $(\nu_k)_{k \in \N} \subset \Delta_a$.
	
	First of all, we need an estimate for the term
	\begin{equation*}
	\bigl| L(\phi_\nu^{\alpha}(t), \chi_\nu^{\alpha}(t), \alpha(t), w) - L(\phi_\rho^{\alpha}(t), \chi_\rho^{\alpha}(t), \alpha(t), w) \bigr|.
	\end{equation*}
	Thanks to the linearity of $L$ in the second argument, it is easy to get that for all $t \in [0, T]$
	\begin{multline*}
	\begin{split}
	&\bigl| L(\phi_\nu^{\alpha}(t), \chi_\nu^{\alpha}(t), \alpha(t), w) - L(\phi_\rho^{\alpha}(t), \chi_\rho^{\alpha}(t), \alpha(t), w) \bigr| \\
	\leq &\bigl| \chi_\nu^{\alpha}(t) - \chi_\rho^{\alpha}(t) \bigr|
	\biggl| \phi_\nu^{\alpha}(t) \bff(\alpha(t)) + r(\phi_\nu^{\alpha}(t), \alpha(t)) \int_{\Delta_e} w(p) R(\phi_\nu^{\alpha}(t), \alpha(t); \dd p) \biggr|
	\end{split} \\
	\qquad + \chi_\rho^{\alpha}(t) 
	\biggl| \biggl[ \phi_\nu^{\alpha}(t) \bff(\alpha(t)) + r(\phi_\nu^{\alpha}(t), \alpha(t)) \int_{\Delta_e} w(p) R(\phi_\nu^{\alpha}(t), \alpha(t); \dd p) \biggr] \\
	\qquad \qquad - \biggl[\phi_\rho^{\alpha}(t) \bff(\alpha(t)) + r(\phi_\rho^{\alpha}(t), \alpha(t))
	\int_{\Delta_e} w(p) R(\phi_\rho^{\alpha}(t), \alpha(t); \dd p) \biggr] \biggr|.
	\end{multline*}
	
	The first summand can be estimated observing that Assumptions \ref{assumption:lambda}, \ref{assumption:costfunction} entail that
	\begin{equation*}
	\biggl| \phi_\nu^{\alpha}(t) \bff(\alpha(t)) + r(\phi_\nu^{\alpha}(t), \alpha(t))\int_{\Delta_e} w(p) R(\phi_\nu^{\alpha}(t), \alpha(t); \dd p) \biggr| \leq K
	\end{equation*}
	where $K > 0$ is a constant depending on $C_f$ and $C_r$ defined in (\ref{eq:fbounded}) and (\ref{eq:rbounded}) and on $\sup_{\theta \in \Delta_e} |w(\theta)|$.
	Moreover, by repeatedly applying Gronwall's Lemma, it can be shown that for all $t \in [0,T]$
	\begin{equation*}
	\bigl| \chi_\nu^{\alpha}(t) - \chi_\rho^{\alpha}(t) \bigr| \leq \frac{L_r}{L_F} (e^{L_F T} - 1) e^{L_r T} |\nu - \rho|
	\end{equation*}
	where $L_F$ is the constant defined in (\ref{eq:FLipschitz}).
	
	As for the second summand, notice that $\chi_\nu^{\alpha}(t) \leq 1$. In addition, Assumption \ref{assumption:costfunction} and Proposition \ref{prop:weakFellerR} imply that there exists a modulus of continuity $\eta \colon [0,+\infty) \to [0,+\infty)$ such that
	\begin{equation*}
	\sup_{u \in U} \biggl|\nu \bff(u) + r(\nu,u)\int_{\Delta_e}w(p) R(\nu,u;\dd p) - \rho \bff(u) - r(\rho,u)\int_{\Delta_e}w(p) R(\rho,u;\dd p)\biggr| 
	\leq \eta(|\nu-\rho|).
	\end{equation*}
	So we have that for all $t \in [0,T]$
	\begin{multline*}
	\biggl| \biggl[ \phi_\nu^{\alpha}(t) \bff(\alpha(t)) + r(\phi_\nu^{\alpha}(t), \alpha(t)) \int_{\Delta_e} w(p) R(\phi_\nu^{\alpha}(t), \alpha(t); \dd p) \biggr] \\
	-\biggl[\phi_\nu^{\alpha}(t) \bff(\alpha(t)) + r(\phi_\nu^{\alpha}(t), \alpha(t))
	\int_{\Delta_e} w(p) R(\phi_\nu^{\alpha}(t), \alpha(t); \dd p) \biggr] \biggr| \\ 
	\leq \eta(|\phi_\nu^{\alpha}(t) - \phi_\rho^{\alpha}(t)|) \leq \eta(|\nu - \rho| e^{L_F T}),
	\end{multline*}
	where the last inequality follows from the fact that $\eta$ is non decreasing and Gronwall's Lemma again.
	
	Collecting all the computations made so far and defining $C = e^{L_F T}$ we get
	\begin{equation*}
	\bigl| L(\phi_\nu^{\alpha}(t), \chi_\nu^{\alpha}(t), \alpha(t), w) - L(\phi_\rho^{\alpha}(t), \chi_\rho^{\alpha}(t), \alpha(t), w) \bigr| \leq 
	\frac{L_r}{L_F} (e^{L_F T} - 1) e^{L_r T} |\nu - \rho| + \eta(C |\nu-\rho|).
	\end{equation*}
	
	We are now in a position to prove our claim. It suffices to notice that
	\begin{multline}\label{eq:continuityJ}
	\bigl|\cJ_{T,w}(\nu, \alpha) - \cJ_{T,w}(\rho, \alpha)\bigr| \\ 
	\leq \int_0^T e^{-\beta t} \biggl| L(\phi_\nu^{\alpha}(t), \chi_\nu^{\alpha}(t), \alpha(t), w) - L(\phi_\rho^{\alpha}(t), \chi_\rho^{\alpha}(t), \alpha(t), w) \biggr| \, \dd t \\
	\leq \frac{e^{-\beta T}-1}{\beta}\biggl[ \frac{L_r}{L_F} (e^{L_F T} - 1) e^{L_r T} |\nu - \rho| + \eta(C |\nu-\rho|) \biggr]
	\end{multline}
	and define $K_1 = \frac{e^{-\beta T}-1}{\beta} \frac{L_r}{L_F} (e^{L_F T} - 1) e^{L_r T}$ and $K_2 = \frac{e^{-\beta T}-1}{\beta}$.
\end{proof}

\begin{proposition} \label{prop:domainG}
	Under Assumptions \ref{assumption:lambda} and \ref{assumption:costfunction}, for each function $w \in \dC(\Delta_e)$ we have that $\cG w \in \dC(\Delta_e)$.
\end{proposition}
\begin{rem}
	In the literature we could only find \citep[Theorem 3.3]{soner:optcontrol1} as a result similar to this one. However, it is not directly applicable to our case. Therefore we provide a complete proof of this Proposition, adapting whenever necessary the arguments of the cited work.
\end{rem}
\begin{proof}
	To start, let us pick $\nu, \rho \in \Delta_a$, $a \in O$, such that for some $\delta > 0$  $|\nu - \rho| < \delta$. Let $\epsilon > 0$, $T > 0$ be arbitrarily fixed and choose $\alpha^\epsilon \in A$ such that
	\begin{equation}\label{eq:alphaepsilon2}
		\cG w(\rho) + \epsilon \geq \int_0^T e^{-\beta t} L(\phi_\rho^{\alpha^\epsilon}(t), \chi_\rho^{\alpha^\epsilon}(t), \alpha^\epsilon(t), w) \, \dd t +
		 e^{-\beta T} \chi_\rho^{\alpha^\epsilon}(T) \cG w(\phi_\rho^{\alpha^\epsilon}(T))
	\end{equation}
	according to the Dynamic Programming Principle.
	We immediately get from (\ref{eq:alphaepsilon2})
	\begin{equation*}
	\begin{split}
		\cG w(\nu) - \cG w(\rho) &\leq \cJ_{T,w}(\nu, \alpha^\epsilon) - \cJ_{T,w}(\rho, \alpha^\epsilon) + \epsilon \\
		&\qquad + e^{-\beta T} \bigl[ \chi_\nu^{\alpha^\epsilon}(T) \cG w(\phi_\nu^{\alpha^\epsilon}(T)) - \chi_\rho^{\alpha^\epsilon}(T) \cG w(\phi_\rho^{\alpha^\epsilon}(T)) \bigr] \\
		&\leq \bigl|\cJ_{T,w}(\nu, \alpha^\epsilon) - \cJ_{T,w}(\rho, \alpha^\epsilon)\bigr| + 
		e^{-\beta T} \bigl|\chi_\nu^{\alpha^\epsilon}(T) - \chi_\rho^{\alpha^\epsilon}(T)\bigr| \sup_{\theta \in \Delta_e} \bigl|\cG w(\theta)\bigr| \\
		&\qquad + e^{-\beta T} \bigl|\cG w(\phi_\nu^{\alpha^\epsilon}(T)) - \cG w(\phi_\rho^{\alpha^\epsilon}(T))\bigr| + \epsilon
	\end{split}
	\end{equation*}
	where $\cJ_{T,w}$ was defined in (\ref{eq:JTwdefinition}) and $\sup_{\theta \in \Delta_e} \bigl|\cG w(\theta)\bigr| < +\infty$ since $w$ is bounded and $\cG$ maps bounded functions into bounded functions.
	
	We need to provide an estimate for the terms appearing in the last lines of the previous equation.
	We know from Lemma \ref{lemma:JTwestimate} that
	\begin{equation*}
		\bigl|\cJ_{T,w}(\nu, \alpha^\epsilon) - \cJ_{T,w}(\rho, \alpha^\epsilon)\bigr| \leq K_1 \delta + K_2 \eta(C \delta)
	\end{equation*}
	where $C, K_1, K_2 > 0$, $\eta$ is a modulus of continuity and it is worth remarking that the estimate is independent of $\alpha^\epsilon$. In particular, $C = e^{L_F T}$.
	Applying Gronwall's Lemma one is able to obtain (see the proof of Lemma \ref{lemma:JTwestimate} for more details) 
	\begin{equation*}
		\bigl|\chi_\nu^{\alpha^\epsilon}(T) - \chi_\rho^{\alpha^\epsilon}(T)\bigr| \leq \frac{K_1}{K_2} \delta.
	\end{equation*}
	As for the term $\bigl|\cG w(\phi_\nu^{\alpha^\epsilon}(T)) - \cG w(\phi_\rho^{\alpha^\epsilon}(T))\bigr|$, let us define for $r > 0$
	\begin{equation*}
		\zeta(r) = 	\suptwo{\nu, \rho \in \Delta_e}{|\nu - \rho| < r} \bigl|\cG w(\nu) - \cG w(\rho)\bigr|
	\end{equation*}
	and set $\zeta(0) = \lim_{r \downarrow 0} \zeta(r)$.
	Since $|\phi_\nu^{\alpha^\epsilon}(T) - \phi_\rho^{\alpha^\epsilon}(T)| \leq C \delta$, we get that
	\begin{equation*}
		\bigl|\cG w(\phi_\nu^{\alpha^\epsilon}(T)) - \cG w(\phi_\rho^{\alpha^\epsilon}(T))\bigr| \leq \zeta(C \delta).
	\end{equation*}
	
	Summarizing all the results obtained so far, we get that for all $\epsilon > 0$
	and all $\nu, \rho \in \Delta_a$, $a \in O$, with $|\nu - \rho| < \delta$,
	\begin{equation*}
		\cG w(\nu) - \cG w(\rho) \leq K_1 \delta + K_2 \eta(C \delta) + e^{-\beta T} \sup_{\theta \in \Delta_e} \bigl|\cG w(\theta)\bigr| \frac{K_1}{K_2} \delta + e^{-\beta T} \zeta(C \delta) + \epsilon.
	\end{equation*}
	Thus, as $\epsilon \to 0^+$ and defining $K_0 = K_1 + e^{-\beta T} \sup_{\theta \in \Delta_e} \bigl|\cG w(\theta)\bigr| \frac{K_1}{K_2}$,
	\begin{equation}\label{eq:zetaestimate}
		\zeta(\delta) \leq K_0 \delta + K_2 \eta(C \delta) + e^{-\beta T} \zeta(C \delta).
	\end{equation}
	
	Now it is left to prove that $\zeta$ is a modulus of continuity for the function $\cG w$ and to do so it suffices to show that $\zeta(0) = 0$.
	Let us choose $\delta = \frac{1}{C^n}$, for some $n \in \N$. Since $C = e^{L_F T} > 1$, proving that $\zeta(0) = 0$ is equivalent to verify that $\lim_{n \to +\infty} \zeta(\frac{1}{C^n}) = 0$, by definition of $\zeta$ in $0$. Assuming, without loss of generality, that $C e^{-\beta T} \neq 1$ and iterating the inequality shown in (\ref{eq:zetaestimate}) we get
	\begin{equation*}
	\begin{split}
		\zeta(0) &\leq \lim_{n \to +\infty} \biggl[
		\frac{K_0}{C^n} \sum_{j = 0}^{n-1} (C e^{-\beta T})^j + 
		K_2 \eta\Bigl(\frac{1}{C^n}\Bigr) \sum_{j = 0}^{n-1} (e^{-\beta T})^j +
		e^{-n\beta T} \zeta(1) \biggr] \\
		&\leq \lim_{n \to +\infty} \biggl[ 
		\frac{1}{1-Ce^{-\beta T}}\Bigl[\frac{K_0}{C^n} - e^{-n\beta T}\Bigr] +
		\frac{K_2}{1-e^{-\beta T}}\eta\Bigl(\frac{1}{C^n}\Bigr) [1-e^{-n\beta T}]\biggr] = 0,
	\end{split}
	\end{equation*}
	hence the desired result.
\end{proof}

We are now in a position to state the first characterization of the PDP value function $v$.
\begin{theorem}
	Under Assumptions \ref{assumption:lambda} and \ref{assumption:costfunction} we have that $v$ is the unique fixed point of the operator $\cG$  in the space of continuous functions on $\Delta_e$.
\end{theorem}
\begin{proof}
	The result follows by combining the fact that $v$ is the unique fixed point of $\cG$ in the space $\dB_b(\Delta_e)$, the fact that the operator $\cG \colon \dC_b(\Delta_e) \to \dC_b(\Delta_e)$ is a contraction mapping and, finally, Proposition \ref{prop:domainG}.
\end{proof}

\subsection{The HJB equation}
Now we move to the second characterization of the PDP value function $v$ in the sense of viscosity solutions.
Using standard arguments of control theory, the Dynamic Programming Principle stated in Proposition \ref{prop:DPP} admits a \emph{local} version in the form the following \mbox{\emph{Hamilton-Jacobi-Bellman}} equation
\begin{equation}\label{eq:HJB}
H(\nu, \dD v(\nu), v) + \beta v(\nu) = 0, \quad \nu \in \Delta_e.
\end{equation}
The function $H \colon \Delta_e \times \R^{\abs{I}} \times \dC(\Delta_e) \to \R$ is called the \emph{hamiltonian} and is defined as
\begin{equation}
H(\nu, \bfb, w) \coloneqq \sup_{u \in U} 
\biggl\{-F(\nu, u)\bfb - \nu \bff(u) - r(\nu, u)
\int_{\Delta_e} \bigl[w(p) - w(\nu)\bigr] R(\nu, u; \dd p) \biggr\}.
\end{equation}

The aim of this subsection to characterize the PDP value function $v$ as the unique \emph{constrained viscosity solution} of the HJB equation (\ref{eq:HJB}).
This concept has been developed by H.~M.~Soner. In \citep{soner:optcontrol1} it is used to characterize the value function of a deterministic optimal control problem with state space constraint; in \citep{soner:optcontrol2} the author extends this definition to study the solution to an \mbox{integro-differential} HJB, associated to an optimal control problem of a PDP with state space constraint.

This approach is particularly well suited to our problem, not only because of the similarities between our situation and the one studied in \citep{soner:optcontrol2}, but also because of the fact that the state space constraint is embedded in our formulation. In fact, the trajectories of the PDP $\bar \pi$ lie in the effective simplex $\Delta_e$ and may as well take values on the boundary of $\Delta_e$. Despite these similarities we will not able to apply directly results of \citep{soner:optcontrol2} to our problem. Some assumptions are not satisfied in our case, \eg Hypothesis (1.3) of that paper, and the proof of the main theorem relies on a slightly different (and somewhat more classical) version of the Dynamic Programming Principle. We will, then, provide a full proof of the following Theorem \ref{th:valuefunctionviscsol} adapting the arguments given in \citep[Th. 1.1]{soner:optcontrol2} as needed. 

First, let us recall the definition of constrained viscosity solution.
In what follows, whenever $K$ is a subset of $\Delta_e$, we will denote by $\bar K$ its relative closure and by $\Int K$ its relative interior. It is understood that all statements referring to topological properties are with respect to the relative topology of $\Delta_e$ as a subset of $\R^{\abs{I}}$ (the latter one equipped with the standard euclidean topology). The set $\dC^1(K)$ will be the set of continuously differentiable real functions on $K$.
\begin{definition}
	A uniformly continuous and bounded function $w \colon \bar K \to \R$ is called a
	\begin{itemize}
		\item \emph{viscosity subsolution} of $H(\nu, \dD w(\nu), w) + \beta w(\nu) = 0$ on $K$ if
		\begin{equation*}
			H(\rho, \dD \psi(\rho), w) + \beta w(\rho) \leq 0
		\end{equation*} 
		whenever $\psi \in \dC^1(N_\rho)$ and $(w - \psi)$ has a global maximum, relative to $K$, at $\rho \in K$,
		where $N_\rho$ is a neighborhood of $\rho$.
		\item \emph{viscosity supersolution} of $H(\nu, \dD w(\nu), w) + \beta w(\nu) = 0$ on $K$ if
		\begin{equation*}
			H(\rho, \dD \psi(\rho), w) + \beta w(\rho) \geq 0
		\end{equation*} 
		whenever $\psi \in \dC^1(N_\rho)$ and $(w - \psi)$ has a global minimum, relative to $K$, at $\rho \in K$,
		where $N_\rho$ is a neighborhood of $\rho$.
		\item \emph{constrained viscosity solution} of $H(\nu, \dD w(\nu), w) + \beta w(\nu) = 0$ on $\bar K$ if it is a subsolution on $K$ and a supersolution on $\bar K$.
	\end{itemize}
\end{definition}
\begin{rem}
	The fact that $w$ is a viscosity supersolution on the closed set $\bar K$ of (\ref{eq:HJB}) automatically imposes a boundary condition. For more details, see the Remark following \citep[Definition 2.1]{soner:optcontrol1}
\end{rem}

Before stating the main Theorem, we need the following lemma. We omit its proof for the reader's convenience. It can be found in \citep[Lemma 2.1]{soner:optcontrol2} (see also Remark 2.1 therein).
\begin{lemma} \label{lemma:hamiltonian}
	Let Assumption \ref{assumption:lambda} hold. A function $w \in \dC(\Delta_e)$ is a viscosity subsolution on $\Int \Delta_e$ (resp. supersolution on $\Delta_e$) of $H(\nu, \dD w(\nu), w) + \beta w(\nu) = 0$ if and only if
	\begin{equation*}
		H(\rho, \dD \psi(\rho), \psi) + \beta w(\rho) \leq (\text{resp. }\geq)\, 0,
	\end{equation*} 
	whenever $\psi \in \dC^1(N_\rho) \cap \dC_b(\Delta_e)$ and $(v-\psi)$ has a global maximum relative to $\Delta_e$ at $\rho \in \Int \Delta_e$ (resp. minimum at $\rho \in \Delta_e$), where $N_\rho$ is a neighborhood of $\rho$.
\end{lemma}

\begin{theorem} \label{th:valuefunctionviscsol}
	Under Assumptions \ref{assumption:lambda} and \ref{assumption:costfunction}, the PDP value function $v$ is the unique constrained viscosity solution of (\ref{eq:HJB}).
\end{theorem}
\begin{proof}
	Uniqueness follows easily from the very same argument given in \citep[Th. 1.1]{soner:optcontrol2}. In fact, the hypothesis labelled as (A1) is satisfied in our framework by each connected component of $\Delta_e$ and other hypotheses are invoked only to show that the functions
	\begin{equation*}
		f_i(\nu, u) = \nu \bff(u) + r(\nu, u)
		\int_{\Delta_e} \bigl[w_i(p) - w_i(\nu)\bigr] R(\nu, u; \dd p),
		\quad \nu \in \Delta_e, u \in U, i = 1,2,
	\end{equation*}
	are uniformly continuous in $\nu$, uniformly with respect to $u$ (here $w_1$, $w_2$ are two arbitrary constrained viscosity solutions of (\ref{eq:HJB}) ).
	This is true in our setting because of Assumption \ref{assumption:costfunction} and Proposition \ref{prop:weakFellerR}.
	Therefore, one can follow the same reasoning to show uniqueness of the solution.
	
	Let us now show that $v$ is a viscosity subsolution on $\Int \Delta_e$ of (\ref{eq:HJB}).
	It is easy to see that in Lemma \ref{lemma:hamiltonian} we can substitute $\psi \in \dC^1(N_\rho) \cap \dC_b(\Delta_e)$ by $\psi \in \dC^1(\Delta_e)$ (see also \citep[Remark 2.1]{soner:optcontrol2}). So, let us fix $\psi \in \dC^1(\Delta_e)$ and $\rho \in \Int \Delta_e$ such that 
	$(v-\psi)(\rho) = \max_{\nu \in \Delta_e}\{(v-\psi)(\nu)\} = 0$.
	Since $v \leq \psi$, from the DPP we get that for all $\alpha \in A$
	\begin{equation} \label{eq:thvisc:DPPinequality}
		v(\rho) = \psi(\rho) \leq \int_0^T e^{-\beta t} L(\phi_\rho^{\alpha}(t), \chi_\rho^{\alpha}(t), \alpha(t), w) \, \dd t + e^{-\beta T} \chi_\rho^{\alpha}(T) \psi(\phi_\rho^{\alpha}(T)).
	\end{equation}
	Differentiating $e^{-\beta t} \chi_\rho^{\alpha}(t) \psi(\phi_\rho^{\alpha}(t))$ we have
	\begin{multline} \label{eq:thvisc:differential}
		\dd\bigl(e^{-\beta t} \chi_\rho^{\alpha}(t) \psi(\phi_\rho^{\alpha}(t))\bigr) =
		e^{-\beta t} \chi_\rho^{\alpha}(t)\\ \bigl\{-\beta \psi(\phi_\rho^{\alpha}(t))
		-r(\phi_\rho^{\alpha}(t), \alpha(t)) \psi(\phi_\rho^{\alpha}(t))
		+F(\phi_\rho^{\alpha}(t), \alpha(t)) \dD \psi(\phi_\rho^{\alpha}(t)) \bigr\} \dd t.
	\end{multline}
	Integrating (\ref{eq:thvisc:differential}) in $[0,T]$ and substituting the result in (\ref{eq:thvisc:DPPinequality}) we obtain
	\begin{multline}
		\int_0^T e^{-\beta t} \chi_\rho^{\alpha}(t) \biggl\{
		\beta \psi(\phi_\rho^{\alpha}(t)) - F(\phi_\rho^{\alpha}(t), \alpha(t)) \dD \psi(\phi_\rho^{\alpha}(t)) \\- \phi_\rho^{\alpha}(t) \bff(\alpha(t))
		-r(\phi_\rho^{\alpha}(t), \alpha(t)) \int_{\Delta_e} \bigl[\psi(p) - \psi(\phi_\rho^{\alpha}(t)) \bigr] R(\phi_\rho^{\alpha}(t), \alpha(t); \dd p) \biggr\} \dd t \leq 0.
	\end{multline}
	By means of Assumption \ref{assumption:costfunction}, Proposition \ref{prop:weakFellerR} and the properties of the flow $\phi_\rho^{\alpha}(\cdot)$, we are able to obtain from the previous inequality the estimate
	\begin{multline*}
		\frac{1}{T} \int_0^T \biggl\{
		\beta \psi(\rho) - F(\rho, \alpha(t)) \dD \psi(\rho) \\- \rho \bff(\alpha(t))
		-r(\rho, \alpha(t)) \int_{\Delta_e} \bigl[\psi(p) - \psi(\rho) \bigr] R(\rho, \alpha(t); \dd p) \biggr\} \dd t \leq h(T)
 	\end{multline*}
 	where $h$ is a continuous function such that $h(0) = 0$.
 	Now, let $t_0 = \frac{\dist(\rho, \partial \Delta_e)}{C_F}$, where $\displaystyle C_F = \sup_{(\nu, u) \in \Delta_e \times U} F(\nu, u)$, so that on $[0, t_0)$ the flow never reaches the boundary of $\Delta_e$. For each fixed $u \in U$ it is clearly possible to pick a control $\alpha \in A$ such that $\alpha(T) = u$, for all $T < t_0$. Using this strategy in the last inequality we get that for all $T \in [0, t_0)$ and all $u \in U$
 	\begin{equation*}
	 	\beta \psi(\rho) - F(\rho, u) \dD \psi(\rho) \\- \rho \bff(u)
	 	-r(\rho, u) \int_{\Delta_e} \bigl[\psi(p) - \psi(\rho) \bigr] R(\rho, u; \dd p) \leq h(T).
 	\end{equation*}
 	Taking the limit as $T \to 0^+$ and the supremum with respect to all $u \in U$ we obtain the subsolution property.
 	
	Let us now show that $v$ is a viscosity supersolution on $\Delta_e$ of (\ref{eq:HJB}).
	Let $\psi \in \dC^1(\Delta_e)$ and $\rho \in \Delta_e$ such that 
	$(v-\psi)(\rho) = \min_{\nu \in \Delta_e}\{(v-\psi)(\nu)\} = 0$.
	Since $v \geq \psi$, from the DPP we get that for all $T > 0$
	\begin{equation}
	v(\rho) = \psi(\rho) \geq \inf_{\alpha \in A} 
	\biggl\{ \int_0^T e^{-\beta t} L(\phi_\rho^{\alpha}(t), \chi_\rho^{\alpha}(t), \alpha(t), w) \, \dd t + e^{-\beta T} \chi_\rho^{\alpha}(T) \psi(\phi_\rho^{\alpha}(T)) \biggr\}.
	\end{equation}
	For each $n \in \N$ consider $T = 1/n$ and pick a control $\alpha^n \in A$ such that
	\begin{equation*}
		\psi(\rho) + \frac{1}{n^2} \geq 
		\int_0^{1/n} e^{-\beta t} L(\phi_\rho^{\alpha}(t), \chi_\rho^{\alpha}(t), \alpha(t), w) \, \dd t + e^{-\beta/n} \chi_\rho^{\alpha}\biggl(\frac{1}{n}\biggr) \psi\biggl(\phi_\rho^{\alpha}\biggl(\frac{1}{n}\biggr)\biggr).	
	\end{equation*}
	With similar computations as before we are able to obtain
	\begin{multline} \label{eq:thvisc:sequence}
		n \int_0^{1/n} \biggl\{
		\beta \psi(\rho) - F(\rho, \alpha^n(t)) \dD \psi(\rho) \\- \rho \bff(\alpha^n(t))
		-r(\rho, \alpha^n(t)) \int_{\Delta_e} \bigl[\psi(p) - \psi(\rho) \bigr] R(\rho, \alpha^n(t); \dd p) \biggr\} \, \dd t \geq h_n
	\end{multline}
	where $h_n \to 0$ as $n \to +\infty$.
	Let us define the following quantities
	\begin{align*}
		F_n &\coloneqq n \int_0^{1/n} F(\rho, \alpha^n(t)) \, \dd t \\
		K_n &\coloneqq n \int_0^{1/n} \biggl\{ \rho \bff(\alpha^n(t))
		+r(\rho, \alpha^n(t)) \int_{\Delta_e} \bigl[\psi(p) - \psi(\rho) \bigr] R(\rho, \alpha^n(t); \dd p) \biggr\} \, \dd t
	\end{align*}
	and the set  $C(\rho) \coloneqq \{(F(\rho, u), \rho \bff(u)
	+r(\rho, u) \int_{\Delta_e} \bigl[\psi(p) - \psi(\rho) \bigr] R(\rho, u; \dd p)), u \in U\}$.
	Notice that $(F_n, K_n) \in \cco C(\rho)$ for all $n \in \N$ and $\cco C(\rho)$ is compact since $C(\rho)$ is bounded. Hence there is a subsequence, still denoted by $(F_n, K_n)$ that converges to some $(F,K) \in \cco C(\rho)$. Therefore, taking the limit as $n$ goes to infinity in (\ref{eq:thvisc:sequence}) we get
	\begin{equation*}
		\beta \psi(\rho) - F \cdot \dD \psi(\rho) - K \geq 0
	\end{equation*}
	so that
	\begin{equation*}
		\beta \psi(\rho) + \sup_{(F, K) \in \cco C(\rho)}\{-F \cdot \dD \psi(\rho) - K\} \geq 0.
	\end{equation*}
	Finally, noticing that
	\begin{equation*}
		\sup_{(F, K) \in \cco C(\rho)}\{-F \cdot \dD \psi(\rho) - K\} = H(\rho, \dD \psi(\rho), \psi)
	\end{equation*}
	we get the desired supersolution property for $v$. 
\end{proof}

\section{Existence of an ordinary optimal control} \label{sec:existenceoptcontrol}
We want now to prove that under some additional assumptions there exists an optimal ordinary control $\bfu^\star \in \cU_{ad}$ such that the minimum in (\ref{eq:valuefunction}) is achieved. Thanks to Theorem \ref{th:costfunctionalidentif} this optimal control exists if and only if there exists an optimal policy $\bfa^\star = (a_0, a_1, \dots) \in \cA_{ad}$ such that for all $n \in \N_0$ the functions $a_n$ take values in the set $A$ of ordinary controls. Since we already established the existence of a stationary optimal policy made of relaxed controls, we want to find an analogous policy made of ordinary controls.

First, we need to find $\alpha^\star \in A$ such that for each fixed $\nu \in \Delta_e$ the functional
\begin{equation} \label{eq:optimalcontrolfunctional}
\cJ(\nu, \alpha) = \int_0^\infty e^{-\beta t} L(\phi_\nu^\alpha(t), \chi_\nu^\alpha(t), \alpha(t), v) \, \dd t
\end{equation}
reaches its infimum (the function $v$ appearing as the last argument of the function $L$ is the PDP value function characterized in the previous section). If this is the case, then an optimal stationary policy $\bfa^\star \in \cA_{ad}$ is granted by standard results in discrete-time control theory, as stated in Theorem \ref{th:optcontrol}. 

\begin{theorem}\label{th:minimizer}
	Let Assumptions \ref{assumption:lambda} and \ref{assumption:costfunction} hold and suppose that for each $\rho \in \Delta_e$ and $s \in [0, 1]$ the set
	\begin{equation*}
	 C(\rho, s) = \{(f, g, l) \in \Delta_e \times [0,1] \times \R \text{ s.t. } 
	 f = F(\rho, u), g = -r(\rho, u)s, l \geq L(\rho, s, u, v), u \in U \}
	\end{equation*}
	is convex.
	
	Then for each fixed $\nu \in \Delta_e$ there exists $\alpha^\star \in A$ such that the infimum of the functional $\cJ$ appearing in (\ref{eq:optimalcontrolfunctional}) is achieved.
\end{theorem}

\begin{proof}
	Fix $\nu \in \Delta_e$ and let us write (\ref{eq:optimalcontrolfunctional}) in a lighter way, suppressing the explicit mention of the value function $v$ and the dependence on the control $\alpha$ and $\nu$ of the functions $\phi_\nu^\alpha$ and $\chi_\nu^\alpha$. We will then write
	\begin{equation*}
	\cJ(\alpha) = \int_0^\infty e^{-\beta t} L(\phi(t), \chi(t), \alpha(t)) \, \dd t.
	\end{equation*}
	
	Let $\alpha_n \in A$, $n \in \N$ be a minimizing sequence for $\cJ$ (\ie $\cJ(\alpha_n) \to \inf_{\alpha} \cJ(\alpha)$ as $n \to +\infty$) and let $(\phi_n, \chi_n)_{n \in \N}$ be the corresponding trajectories of the flow and the survival distribution of the first jump time of the PDP.
	For each $n \in \N$, $\phi_n \in \dC([0, +\infty); \Delta_e)$ and $\chi_n \in \dC([0, +\infty); [0,1])$. It can be easily checked that both sequences $(\phi_n)_{n \in \N}$ and $(\chi_n)_{n \in \N}$ are uniformly bounded and equicontinuous on each compact subset of $[0, +\infty)$, hence by \mbox{Ascoli-Arzel\`a} theorem we get that there exist
	$\phi \in \dC([0, +\infty); \Delta_e)$ and $\chi \in \dC([0, +\infty); [0,1])$ such that, up to a subsequence, $\phi_n \to \phi$ and $\chi_n \to \chi$ uniformly on each compact subset of $[0, +\infty)$.
	
	Let us now define for all $t \geq 0$
	\begin{itemize}
		\item $F_n(t) = F(\phi_n(t), \alpha_n(t))$,
		\item $G_n(t) = -r(\phi_n(t), \alpha_n(t)) \chi_n(t)$,
		\item $L_n(t) = L(\phi_n(t), \chi_n(t), \alpha_n(t))$.
	\end{itemize}
	Denoting by $\dL_\beta^1$ the weighted $\dL^1$ space (with weight given by the discount factor $\beta$), it can be easily shown that, for each $n \in \N$, $F_n \in \dL_\beta^1([0, +\infty); \Delta_e)$, $G_n \in \dL_\beta^1([0, +\infty); \R)$, $L_n \in \dL_\beta^1([0, +\infty); \R)$ and that the three sequences $(F_n)_{n \in \N}$, $(G_n)_{n \in \N}$ and $(L_n)_{n \in \N}$ are uniformly bounded and uniformly integrable. Hence there exist $\hat F \in \dL_\beta^1([0, +\infty); \Delta_e)$, $\hat G \in \dL_\beta^1([0, +\infty); \R)$ and $\hat L \in \dL_\beta^1([0, +\infty); \R)$ such that, up to a subsequence, $F_n \rightharpoonup \hat F$, $G_n \rightharpoonup \hat G$ and $L_n \rightharpoonup \hat L$ weakly in $\dL_\beta^1$.
	
	By Mazur's Theorem (see \eg \citep[Corollary 3.8, p. 61]{brezis:functionalanalysis}, or \citep[Theorem 2, p. 120]{yosida:functionalanalysis}), there exist sequences, still denoted by $(F_n)_{n \in \N}$, $(G_n)_{n \in \N}$ and $(L_n)_{n \in \N}$, that are convex combinations of the elements of the original ones, such that $F_n \rightarrow \hat F$, $G_n \rightarrow \hat G$ and $L_n \rightarrow \hat L$ strongly in $\dL_\beta^1$ and also, again up to a subsequence, a.e. in $[0, +\infty)$.
	Thanks to the hypotheses we have that the functions $F$, $L$ and $-r(\rho, u)s$ are continuous on the compact set $\Delta_e \times [0,1] \times U$ and it can be proved that the sets $C(\rho, s)$ are closed for each $\rho \in \Delta_e$ and $s \in [0,1]$ (see \eg \citep[8.5.vi, p. 296]{cesari:optimization}). Therefore, for almost all $t \geq 0$ the triple $(\hat F(t), \hat G(t), \hat L(t))$ belongs to the set $C(\phi(t), \chi(t))$ and we can apply standard measurable selection theorems (see \eg \citep[8.2.ii, p. 277]{cesari:optimization}, or \citep[Corollary 2.26, p. 102]{liyong:optimalcontroltheory}) to obtain a measurable function $\alpha^\star$ such that
	\begin{itemize}
		\item $\hat F(t) = F(\phi(t), \alpha^\star(t))$,
		\item $\hat G(t) = -r(\phi(t), \alpha^\star(t)) \chi(t)$,
		\item $\hat L(t) = L(\phi(t), \chi(t), \alpha^\star(t)) + z(t)$,
	\end{itemize}
	where $z$ is a \mbox{non-negative} function defined on $[0, +\infty)$.
	
	Now it remains to prove that $\alpha^\star$ is optimal for the functional $\cJ$.
	Let $(\gamma_{k, n})$, where $n \in \N$ and $k \geq n$, be the system of \mbox{non-negative} numbers of Mazur's Theorem, such that for each $n \in \N$
	\begin{align}
		\sum_{k = n}^{K_n} \gamma_{k, n} = 1, & & 
		L(\phi(t), \chi(t), \alpha^\star(t)) = 
		\lim_{n \to +\infty} \sum_{k = n}^{K_n} \gamma_{k, n} L(\phi_k(t), \chi_k(t), \alpha_k(t)).
	\end{align}
	First of all, let us notice that $z$ has to be zero a.e. in $[0, +\infty)$. If this were not the case, we would reach a contradiction (arguing as in the following lines) with the fact that $(\alpha_n)_{n \in \N}$ is a minimizing sequence for $\cJ$.
	Since the function $L$ is bounded by some constant $K > 0$ and obviously the function $Ke^{-\beta t} \in \dL^1([0, +\infty))$, we can apply Fatou's Lemma to obtain
	\begin{equation}
	\begin{split}
		\cJ(\alpha^\star) &= \int_0^\infty e^{-\beta t} L(\phi(t), \chi(t), \alpha^\star(t)) \\ 
		&\leq \liminf_{n \to +\infty} \sum_{k = n}^{K_n} \gamma_{k n} \int_0^\infty e^{-\beta t} L(\phi_k(t), \chi_k(t), \alpha_k(t)) \\
		&= \liminf_{n \to +\infty} \sum_{k = n}^{K_n} \gamma_{k n} \cJ(\alpha_k) = \inf_{\alpha} \cJ(\alpha).
	\end{split}
	\end{equation}
	The claim follows since clearly $\inf_{\alpha} \cJ(\alpha) \leq \cJ(\alpha^\star)$.
\end{proof}

\begin{rem}
	Convexity of the sets $C(\rho, s)$ is guaranteed, for instance, when
	\begin{itemize}
		\item $U \subset \R$ is a closed interval.
		\item Matrix coefficients $\lambda_{ij}(u)$ are linear in $u$, for all $i, j \in I$, $i \ne j$.
		\item The functions $u \mapsto f(i, u)$ are convex for each $i \in I$ .
	\end{itemize}
\end{rem}

We are now ready to state the main result of this Section. To be fully precise its proof would require to formulate the entire control problem in a broader setting. This should be done to allow for more general control policies, namely universally measurable ones. However, this formulation does not pose any particular problem (the interested reader may consult \citep{bertsekas:stochoptcontrol}) and it is irrelevant to the results of this paper. Therefore, we will omit all unnecessary technical details.

\begin{theorem}\label{th:optcontrol}
	For each initial law $\mu \in \Delta$ there exists an optimal ordinary stationary policy $\bfa^\star \in \cA_{ad}$ (with corresponding optimal ordinary control $\bfu^\star \in \cU_{ad}$), \ie an admissible policy with values in the set of ordinary controls $A$ such that
	\begin{equation*}
		V(\mu) = J(\mu, \bfu^\star) = \sum_{a \in O} \mu\bigl(\pre(a)\bigr) \bar J(H_a[\mu], \bfa^\star) = \sum_{a \in O} \mu\bigl(\pre(a)\bigr) v(H_a[\mu]).
	\end{equation*}.
\end{theorem}
\begin{proof}
	Let $\mu \in \Delta$ be fixed. Thanks to Theorem \ref{th:minimizer}, to the fact that the function $\cJ$ appearing in (\ref{eq:optimalcontrolfunctional}) is measurable and to the fact that $\Delta_e$ and $A$ are Borel spaces, standard selection theorems (see \eg \citep[Prop. 7.50]{bertsekas:stochoptcontrol}) ensure that there exists a universally measurable selector $a^u \colon \Delta_e \to A$ such that for all $\nu \in \Delta_e$
	\begin{equation*}
		v(\nu) = \cJ\bigl(\nu, a^u(\nu)\bigr) = \inf_{\alpha \in A} \cJ(\nu, \alpha).
	\end{equation*}
	Let $Q$ be the probability measure on $\Delta_e$ defined in (\ref{eq:probabilityQ}) and let us define the optimal strategy $\bfa^u = (a^u, a^u, \dots)$. Thanks to \citep[Th. 3.1]{yushkevich:controlledjump} we can conclude that there exists a stationary policy $\bfa^\star \in \cA_{ad}$ such that
	\begin{equation*}
		\bar J(\cdot, \bfa^u) = \bar J(\cdot, \bfa^\star) \quad Q-\text{a.s.}.
	\end{equation*}
	Since $Q$ is concentrated at points $\{H_a[\mu]\}_{a \in O}$ we get that for all $a \in O$
	\begin{equation*}
		v(H_a[\mu]) = \bar J(H_a[\mu], \bfa^u) = \bar J(H_a[\mu], \bfa^\star)
	\end{equation*}
	and the claim follows immediately.
\end{proof}

\textit{Acknowledgments.} The author wishes to thank Elsa Maria Marchini for useful discussions on the last section of this paper.

\bibliographystyle{plainnat} 
\bibliography{Bibliography}

\begin{thebibliography}{27}
\providecommand{\natexlab}[1]{#1}
\providecommand{\url}[1]{\texttt{#1}}
\expandafter\ifx\csname urlstyle\endcsname\relax
  \providecommand{\doi}[1]{doi: #1}\else
  \providecommand{\doi}{doi: \begingroup \urlstyle{rm}\Url}\fi

\bibitem[Asmussen(2003)]{asmussen:applprob}
S.~Asmussen.
\newblock \emph{Applied probability and queues}, volume~51 of
  \emph{Applications of Mathematics (New York)}.
\newblock Springer-Verlag, New York, second edition, 2003.
\newblock Stochastic Modelling and Applied Probability.

\bibitem[Bensoussan et~al.(2005)Bensoussan, {\c{C}}akany{\i}ld{\i}r{\i}m, and
  Sethi]{bensoussan:inventory}
A.~Bensoussan, M.~{\c{C}}akany{\i}ld{\i}r{\i}m, and S.~P. Sethi.
\newblock On the optimal control of partially observed inventory systems.
\newblock \emph{C. R. Math. Acad. Sci. Paris}, 341\penalty0 (7):\penalty0
  419--426, 2005.

\bibitem[Bertsekas and Shreve(1978)]{bertsekas:stochoptcontrol}
D.~P. Bertsekas and S.~E. Shreve.
\newblock \emph{Stochastic optimal control. The discrete time case}, volume 139
  of \emph{Mathematics in Science and Engineering}.
\newblock Academic Press, Inc., New York-London, 1978.

\bibitem[Brezis(2011)]{brezis:functionalanalysis}
H.~Brezis.
\newblock \emph{Functional analysis, {S}obolev spaces and partial differential
  equations}.
\newblock Universitext. Springer, New York, 2011.

\bibitem[Brémaud(1981)]{bremaud:pp}
P.~Brémaud.
\newblock \emph{Point Processes and Queues}.
\newblock Springer Series in Statistics. Springer-Verlag, New York, 1981.

\bibitem[Bryson and Johansen(1965)]{bryson:linearfilt}
A.~E. Bryson, Jr. and D.~E. Johansen.
\newblock Linear filtering for time-varying systems using measurements
  containing colored noise.
\newblock \emph{IEEE Trans. Automatic Control}, AC-10:\penalty0 4--10, 1965.

\bibitem[Cesari(1983)]{cesari:optimization}
L.~Cesari.
\newblock \emph{Optimization theory and applications}, volume~17 of
  \emph{Applications of Mathematics (New York)}.
\newblock Springer-Verlag, New York, 1983.
\newblock Problems with ordinary differential equations.

\bibitem[Confortola and Fuhrman(2013)]{confortola:filt}
F.~Confortola and M.~Fuhrman.
\newblock Filtering of continuous-time {M}arkov chains with noise-free
  observation and applications.
\newblock \emph{Stochastics An International Journal of Probability and
  Stochastic Processes}, 85\penalty0 (2):\penalty0 216--251, 2013.

\bibitem[Costa et~al.(2016)Costa, Dufour, and
  Piunovskiy]{costadufour:PDPoptcontrol}
O.~L.~V. Costa, F.~Dufour, and A.~B. Piunovskiy.
\newblock Constrained and unconstrained optimal discounted control of piecewise
  deterministic {M}arkov processes.
\newblock \emph{SIAM J. Control Optim.}, 54\penalty0 (3):\penalty0 1444--1474,
  2016.

\bibitem[Crisan et~al.(2009)Crisan, Kouritzin, and Xiong]{crisan:nonlinfilt}
D.~Crisan, M.~Kouritzin, and J.~Xiong.
\newblock Nonlinear filtering with signal dependent observation noise.
\newblock \emph{Electron. J. Probab.}, 14:\penalty0 no. 63, 1863--1883, 2009.

\bibitem[Davis(1993)]{davis:markovmodels}
M.H.A. Davis.
\newblock \emph{Markov Models and Optimization}, volume~49 of \emph{Monographs
  on Statistics and Applied Probability}.
\newblock Chapman and Hall, London, 1993.

\bibitem[Elliott et~al.(1995)Elliott, Aggoun, and Moore]{elliott:hmm}
R.~J. Elliott, L.~Aggoun, and J.~B. Moore.
\newblock \emph{Hidden {M}arkov models}, volume~29 of \emph{Applications of
  Mathematics (New York)}.
\newblock Springer-Verlag, New York, 1995.
\newblock Estimation and control.

\bibitem[Jacod(1974/75)]{jacod:mpp}
J.~Jacod.
\newblock Multivariate point processes: predictable projection,
  {R}adon-{N}ikod\'ym derivatives, representation of martingales.
\newblock \emph{Z. Wahrscheinlichkeitstheorie und Verw. Gebiete}, 31:\penalty0
  235--253, 1974/75.

\bibitem[Joannides and LeGland(1997)]{joannides:nonlinfilt}
M.~Joannides and F.~LeGland.
\newblock Nonlinear filtering with continuous time perfect observations and
  noninformative quadratic variation.
\newblock In \emph{Proceeding of the 36th IEEE Conference on Decision and
  Control}, pages 1645--1650, 1997.

\bibitem[K{\"o}rezlio{\u{g}}lu and Runggaldier(1993)]{runggaldier:filt}
H.~K{\"o}rezlio{\u{g}}lu and W.~J. Runggaldier.
\newblock Filtering for nonlinear systems driven by nonwhite noises: an
  approximation scheme.
\newblock \emph{Stochastics Stochastics Rep.}, 44\penalty0 (1-2):\penalty0
  65--102, 1993.

\bibitem[Li and Yong(1995)]{liyong:optimalcontroltheory}
X.~Li and J.~Yong.
\newblock \emph{Optimal control theory for infinite-dimensional systems}.
\newblock Systems \& Control: Foundations \& Applications. Birkh\"auser Boston,
  Inc., Boston, MA, 1995.

\bibitem[Norris(1998)]{norris:markovchains}
J.~R. Norris.
\newblock \emph{Markov chains}, volume~2 of \emph{Cambridge Series in
  Statistical and Probabilistic Mathematics}.
\newblock Cambridge University Press, Cambridge, 1998.

\bibitem[Sainte-Beuve(1978)]{sainte-beuve:relaxeddensity}
M.-F. Sainte-Beuve.
\newblock Some topological properties of vector measures with bounded variation
  and its applications.
\newblock \emph{Ann. Mat. Pura Appl. (4)}, 116:\penalty0 317--379, 1978.

\bibitem[Soner(1986{\natexlab{a}})]{soner:optcontrol1}
H.~M. Soner.
\newblock Optimal control with state-space constraint. {I}.
\newblock \emph{SIAM J. Control Optim.}, 24\penalty0 (3):\penalty0 552--561,
  1986{\natexlab{a}}.

\bibitem[Soner(1986{\natexlab{b}})]{soner:optcontrol2}
H.~M. Soner.
\newblock Optimal control with state-space constraint. {II}.
\newblock \emph{SIAM J. Control Optim.}, 24\penalty0 (6):\penalty0 1110--1122,
  1986{\natexlab{b}}.

\bibitem[Takeuchi and Akashi(1985)]{takeuchi:lsqestimation}
Y.~Takeuchi and H.~Akashi.
\newblock Least-squares state estimation of systems with state-dependent
  observation noise.
\newblock \emph{Automatica J. IFAC}, 21\penalty0 (3):\penalty0 303--313, 1985.

\bibitem[Vermes(1985)]{vermes:optcontrol}
D.~Vermes.
\newblock Optimal control of piecewise deterministic {M}arkov process.
\newblock \emph{Stochastics}, 14\penalty0 (3):\penalty0 165--207, 1985.

\bibitem[Winter(2008)]{winter:phdthesis}
J.~T. Winter.
\newblock \emph{Optimal control of markovian jump processes with different
  information structures}.
\newblock PhD thesis, Universit{\"a}t {U}lm, 2008.

\bibitem[Xiong(2008)]{xiong:stochfilt}
J.~Xiong.
\newblock \emph{An introduction to stochastic filtering theory}, volume~18 of
  \emph{Oxford Graduate Texts in Mathematics}.
\newblock Oxford University Press, Oxford, 2008.

\bibitem[Yosida(1980)]{yosida:functionalanalysis}
K.~Yosida.
\newblock \emph{Functional analysis}, volume 123 of \emph{Grundlehren der
  Mathematischen Wissenschaften [Fundamental Principles of Mathematical
  Sciences]}.
\newblock Springer-Verlag, Berlin-New York, sixth edition, 1980.

\bibitem[Yushkevich(1980)]{yushkevich:jumpmodel}
A.~A. Yushkevich.
\newblock On reducing a jump controllable {M}arkov model to a model with
  discrete time.
\newblock \emph{Theory Probab. Appl.}, 25\penalty0 (1):\penalty0 58--69, 1980.

\bibitem[Yushkevich(1981)]{yushkevich:controlledjump}
A.~A. Yushkevich.
\newblock Controlled jump markov models.
\newblock \emph{Theory of Probability \& Its Applications}, 25\penalty0
  (2):\penalty0 244--266, 1981.

\end{thebibliography}
 
\end{document}